\documentclass[english,oneside,a4paper,10pt]{article}
\usepackage{amssymb}
\usepackage{latexsym}
\usepackage[english]{babel}
\usepackage{amsmath}
\usepackage{amsfonts}
\usepackage{amsbsy}
\usepackage[mathscr]{eucal}
\usepackage[latin1]{inputenc}
\usepackage{verbatim}
\usepackage{mathrsfs}
\usepackage{graphicx}
\usepackage{float}
\usepackage{bbm}
\usepackage{lscape}
\usepackage{mathdots}
\usepackage{empheq}
\usepackage[dvipsnames]{xcolor}
\usepackage{array}
\usepackage{a4wide}
\usepackage[dvips,colorlinks,bookmarks,breaklinks]{hyperref}  
\hypersetup{linkcolor=black,citecolor=black,filecolor=black,urlcolor=black}

\providecommand{\norm}[1]{\left\lVert#1 \right\rVert}

\newtheorem{thr}{Theorem}

\newtheorem{lem}{Lemma}

\newenvironment{rem}
{\begin{trivlist}\item[\hskip%
\labelsep{{\it \noindent Remark}}]}{\hfill
\end{trivlist}}

\newenvironment{proof}
{\begin{trivlist}\item[\hskip%
\labelsep{\it \noindent Proof.}]}{\hfill $\square$ \rm
\end{trivlist}}

\newenvironment{acknowledgement}
{\begin{trivlist}\item[\hskip%
\labelsep{\bf \noindent Acknowledgements.}]}{\hfill
\end{trivlist}}

\numberwithin{equation}{section}

\allowdisplaybreaks

\begin{document}
\begin{center}
{\huge {\bf Spectral properties for a type of} \\ {\bf heptadiagonal symmetric matrices}} \\
\vspace{1.0cm}
{\Large João Lita da Silva\footnote{\textit{E-mail address:} \texttt{jfls@fct.unl.pt}}} \\
\vspace{0.1cm}
\textit{Department of Mathematics and GeoBioTec \\ Faculty of Sciences and Technology \\
NOVA University of Lisbon \\ Quinta da Torre, 2829-516 Caparica,
Portugal}
\end{center}

\bigskip

\bigskip

\bigskip

\begin{abstract}
In this paper we express the eigenvalues of a sort of real heptadiagonal symmetric matrices as the zeros of explicit rational functions establishing upper and lower bounds for each of them. From these prescribed eigenvalues we compute also eigenvectors for these type of matrices. A formula not depending on any unknown parameter for the determinant and the inverse of these heptadiagonal matrices is still provided.
\end{abstract}

\bigskip

{\small{\textit{Key words:} Heptadiagonal matrix, eigenvalue, eigenvector, determinant, inverse matrix}}

\bigskip

{\small{\textbf{2010 Mathematics Subject Classification:}
15A18, 15A15, 15A09.}}

\bigskip

\section{Introduction}\label{sec:1}

\indent

The main goal of this paper is to express the eigenvalues of the following $n \times n$ real heptadiagonal matrix
\begin{equation}\label{eq:1.1}
\mathbf{H}_{n} = \left[
\begin{array}{ccccccccccc}
  \xi & \eta & c & d & 0 & \ldots & \ldots & \ldots & \ldots & \ldots & 0 \\
  \eta & a & b & c & d &  \ddots & & & & & \vdots \\
  c & b & a & b & c & \ddots & \ddots &  & &  & \vdots \\
  d & c & b & a & b & \ddots & \ddots & \ddots & & & \vdots \\
  0 & d & c & b & a & \ddots & \ddots & \ddots & \ddots &  & \vdots \\
  \vdots & \ddots & \ddots & \ddots & \ddots & \ddots & \ddots & \ddots & \ddots & \ddots & \vdots \\
  \vdots & & \ddots & \ddots & \ddots & \ddots & a & b & c & d & 0 \\
  \vdots &  & & \ddots & \ddots & \ddots & b & a & b & c & d \\
  \vdots &  & & & \ddots & \ddots & c & b & a & b & c \\
  \vdots &  &  & & & \ddots & d & c & b & a & \eta \\
  0 & \ldots & \ldots & \ldots & \ldots & \ldots & 0 & d & c & \eta & \xi
\end{array}
\right]
\end{equation}
as the zeros of explicit rational functions giving also upper and lower bounds non-depending of any unknown parameter to each of them. Further, we shall compute eigenvectors for these sort of matrices at the expense of the prescribed eigenvalues. The matrices of the form \eqref{eq:1.1} fall into a general class of matrices called band matrices (see \cite{Pissanetsky84}, page $13$) which are widely used in several areas of science and engineering such as numerical solution of ordinary and partial differential equations (ODE and PDE), interpolation problems, boundary value problems among others (see, for instance, \cite{Asplund59}, \cite{Demko77}, \cite{Fischer69}, \cite{Haley80}, \cite{Keeping70}, \cite{Usmani76}).

To accomplish our purpose and in a first stage, we shall exploit the so-called \emph{modification technique} founded by Fasino in \cite{Fasino88} for matrices of the type \eqref{eq:1.1} in order to decompose them into an orthogonal block diagonalization and, at a second stage, use results concerning to a secular equation of diagonal matrices perturbed by the addition of rank-one matrices developed by Anderson in the nineties (see \cite{Anderson96}). Our decomposition will also lead us to explicit formulae for the determinant and the inverse of complex heptadiagonal matrices \eqref{eq:1.1} assuming, of course, its nonsingularity.

\bigskip

\section{Auxiliary tools}\label{sec:2}

\indent

Consider the class of matrices defined by
\begin{equation*}
\EuScript{A}_{n} := \left\{\mathbf{A}_{n} \in \EuScript{M}_{n}(\mathbb{C}): \left[\mathbf{A}_{n} \right]_{k-1,\ell} + \left[\mathbf{A}_{n} \right]_{k+1,\ell} = \left[\mathbf{A}_{n} \right]_{k,\ell-1} + \left[\mathbf{A}_{n} \right]_{k,\ell+1}, \, k,\ell = 1,\ldots,n \right\}
\end{equation*}
where it is assumed $\left[\mathbf{A}_{n} \right]_{n+1,\ell} = \left[\mathbf{A}_{n} \right]_{k,n+1} = \left[\mathbf{A}_{n} \right]_{0,\ell} = \left[\mathbf{A}_{n} \right]_{k,0} = 0$ for all $k,\ell$.
We begin by gather two results announced in \cite{Bini83}, presenting them for complex matrices. The proofs do not suffer any changes from the original ones and so we omit the details.


\begin{lem}\label{lem1}
\noindent \newline \textnormal{(a)} The class $\EuScript{A}_{n}$ is the algebra generated over $\mathbb{C}$ by the $n \times n$ matrix
\begin{equation}\label{eq:2.0}
\boldsymbol{\Omega}_{n} = \left[
\begin{array}{ccccccc}
  0 & 1 & 0 & \ldots & \ldots & \ldots & 0 \\
  1 & 0 & 1 & \ddots & & & \vdots \\
  0 & 1 & 0 & \ddots & \ddots & & \vdots \\
  \vdots & \ddots & \ddots & \ddots & \ddots & \ddots & \vdots \\
  \vdots & & \ddots & \ddots & 0 & 1 & 0 \\
  \vdots & & & \ddots & 1 & 0 & 1 \\
  0 & \ldots & \ldots & \ldots & 0 & 1 & 0
\end{array}
\right].
\end{equation}

\medskip

\noindent \textnormal{(b)} If $\mathbf{A}_{n} \in \EuScript{A}_{n}$ and $\mathbf{a}^{\top}$ is its first row then
\begin{equation*}
\mathbf{A}_{n} = \sum_{k=0}^{n-1} \omega_{k+1} \boldsymbol{\Omega}_{n}^{k}
\end{equation*}
where $\boldsymbol{\Omega}_{n}$ is the $n \times n$ matrix \eqref{eq:2.0} and $\boldsymbol{\omega}^{\top} = (\omega_{1},\omega_{2},\ldots,\omega_{n})$ is the solution of the upper triangular system $\mathbf{U}_{n} \boldsymbol{\omega} = \mathbf{a}$, with $\left[\mathbf{U}_{n} \right]_{0,0} := 1$, $\left[\mathbf{U}_{n} \right]_{0,\ell} := 0$ for all $1 \leqslant \ell \leqslant n$,
\begin{equation*}
\left[\mathbf{U}_{n} \right]_{k,\ell} := \left\{
\begin{array}{l}
  \left[\mathbf{U}_{n} \right]_{k-1,\ell-1} + \left[\mathbf{U}_{n} \right]_{k+1,\ell-1}, \quad 1 \leqslant k \leqslant \ell \leqslant n \\
  0, \quad \textit{otherwise}
\end{array}
\right..
\end{equation*}
\end{lem}

Throughout, $n$ will denote an integer greater or equal to $5$ and $\mathbf{S}_{n}$ will be the $n \times n$ symmetric, involutory and orthogonal matrix defined by
\begin{equation}\label{eq:2.1}
\left[\mathbf{S}_{n} \right]_{k,\ell} := \sqrt{\frac{2}{n+1}} \sin \left(\frac{k \ell \pi}{n+1} \right).
\end{equation}
Our second auxiliary result provide us an orthogonal diagonalization for the following $n \times n$ complex heptadiagonal symmetric matrix
\begin{equation}\label{eq:2.2}
\mathbf{\widehat{H}}_{n} = \left[
\begin{array}{ccccccccccc}
  a-c & b-d & c & d & 0 & \ldots & \ldots & \ldots & \ldots & \ldots & 0 \\
  b-d & a & b & c & d &  \ddots & & & & & \vdots \\
  c & b & a & b & c & \ddots & \ddots &  & &  & \vdots \\
  d & c & b & a & b & \ddots & \ddots & \ddots & & & \vdots \\
  0 & d & c & b & a & \ddots & \ddots & \ddots & \ddots &  & \vdots \\
  \vdots & \ddots & \ddots & \ddots & \ddots & \ddots & \ddots & \ddots & \ddots & \ddots & \vdots \\
  \vdots & & \ddots & \ddots & \ddots & \ddots & a & b & c & d & 0 \\
  \vdots &  & & \ddots & \ddots & \ddots & b & a & b & c & d \\
  \vdots &  & & & \ddots & \ddots & c & b & a & b & c \\
  \vdots &  &  & & & \ddots & d & c & b & a & b-d \\
  0 & \ldots & \ldots & \ldots & \ldots & \ldots & 0 & d & c & b-d & a-c
\end{array}
\right].
\end{equation}


\begin{lem}\label{lem2}
Let $n \geqslant 5$ be an integer, $a,b,c,d \in \mathbb{C}$ and
\begin{equation}\label{eq:2.3}
\lambda_{k} = a + 2b \cos \left(\tfrac{k \pi}{n+1} \right) + 2c \cos \left(\tfrac{2 k \pi}{n+1} \right) + 2d \cos \left(\tfrac{3 k \pi}{n+1} \right), \quad k = 1,\ldots,n.
\end{equation}
If $\mathbf{\widehat{H}}_{n}$ is the $n \times n$ matrix \eqref{eq:2.2} then
\begin{equation*}
\mathbf{\widehat{H}}_{n} = \mathbf{S}_{n} \boldsymbol{\Lambda}_{n} \mathbf{S}_{n}
\end{equation*}
where $\boldsymbol{\Lambda}_{n} = \mathrm{diag} \left(\lambda_{1},\ldots, \lambda_{n} \right)$ and $\mathbf{S}_{n}$ is the matrix \eqref{eq:2.1}.
\end{lem}

\begin{proof}
Suppose an integer $n \geqslant 5$ and $a,b,c,d \in \mathbb{C}$. Since $\mathbf{\widehat{H}}_{n} \in \EuScript{A}_{n}$ and its first row is $\mathbf{a}^{\top} = (a - c, b - d,c,d,0,\ldots,0)$ we have, from Lemma~\ref{lem1},
\begin{equation*}
\mathbf{\widehat{H}}_{n} = (a - 2c) \mathbf{I}_{n} + (b - 3d) \boldsymbol{\Omega}_{n} + c \, \boldsymbol{\Omega}_{n}^{2} + d \, \boldsymbol{\Omega}_{n}^{3}.
\end{equation*}
Using the spectral decomposition
\begin{equation*}
\boldsymbol{\Omega}_{n} = \sum_{\ell=1}^{n} 2 \cos \left(\frac{\ell \pi}{n + 1} \right) \mathbf{s}_{\ell} \, \mathbf{s}_{\ell}^{\top},
\end{equation*}
where
\begin{equation*}
\mathbf{s}_{\ell} = \left[
\begin{array}{c}
  \sqrt{\frac{2}{n+1}} \sin \left(\frac{\ell \pi}{n + 1} \right) \\
  \sqrt{\frac{2}{n+1}} \sin \left(\frac{2\ell \pi}{n + 1} \right) \\
  \vdots \\
  \sqrt{\frac{2}{n+1}} \sin \left(\frac{n\ell \pi}{n + 1} \right)
\end{array}
\right]
\end{equation*}
(i.e. the $\ell$th column of $\mathbf{S}_{n}$), it follows
\begin{align*}
\mathbf{\widehat{H}}_{n} &= \sum_{\ell=1}^{n} \left[(a - 2c) + 2 (b - 3d) \cos \left(\frac{\ell \pi}{n + 1} \right) + 4c \cos^{2} \left(\frac{\ell \pi}{n + 1} \right) + 8d \cos^{3} \left(\frac{\ell \pi}{n + 1} \right) \right] \mathbf{s}_{\ell} \, \mathbf{s}_{\ell}^{\top} \\
&= \sum_{\ell=1}^{n} \lambda_{k} \mathbf{s}_{\ell} \, \mathbf{s}_{\ell}^{\top} \\
&= \mathbf{S}_{n} \boldsymbol{\Lambda}_{n} \mathbf{S}_{n}
\end{align*}
where $\boldsymbol{\Lambda}_{n} = \mathrm{diag} \left(\lambda_{1},\ldots, \lambda_{n} \right)$ and $\mathbf{S}_{n}$ is the matrix \eqref{eq:2.1}. The proof is completed.
\end{proof}

The following statement is an orthogonal block diagonalization for matrices $\mathbf{H}_{n}$ of the form \eqref{eq:1.1} extending Proposition 3.1 in \cite{Bini83} which is valid only for heptadiagonal symmetric Toeplitz matrices.


\begin{lem}\label{lem3}
Let $n \geqslant 5$ be an integer, $a,b,c,d, \xi, \eta \in \mathbb{C}$, $\lambda_{k}$, $k = 1,\ldots,n$ be given by \eqref{eq:2.3} and $\mathbf{H}_{n}$ the $n \times n$ matrix \eqref{eq:1.1}.

\medskip

\noindent \textnormal{(a)} If $n$ is even,
\begin{subequations}
\begin{equation}\label{eq:2.4a}
\mathbf{x} = \left[
\begin{array}{c}
  \frac{2}{\sqrt{n + 1}} \sin\left(\frac{\pi}{n+1} \right) \\[5pt]
  \frac{2}{\sqrt{n + 1}} \sin\left(\frac{3\pi}{n+1} \right) \\[5pt]
  \vdots \\[5pt]
  \frac{2}{\sqrt{n + 1}} \sin\left[\frac{(n - 1)\pi}{n+1} \right]
\end{array}
\right], \quad
\mathbf{y} = \left[
\begin{array}{c}
  \frac{2}{\sqrt{n + 1}} \sin\left(\frac{2\pi}{n+1} \right) \\[5pt]
  \frac{2}{\sqrt{n + 1}} \sin\left(\frac{6\pi}{n+1} \right) \\[5pt]
  \vdots \\[5pt]
  \frac{2}{\sqrt{n + 1}} \sin\left[\frac{(2n - 2)\pi}{n+1} \right]
\end{array}
\right]
\end{equation}
and
\begin{equation}\label{eq:2.4b}
\mathbf{v} = \left[
\begin{array}{c}
  \frac{2}{\sqrt{n + 1}} \sin\left(\frac{2\pi}{n+1} \right) \\[5pt]
  \frac{2}{\sqrt{n + 1}} \sin\left(\frac{4\pi}{n+1} \right) \\[5pt]
  \vdots \\[5pt]
  \frac{2}{\sqrt{n + 1}} \sin\left(\frac{n\pi}{n+1} \right)
\end{array}
\right], \quad
\mathbf{w} = \left[
\begin{array}{c}
  \frac{2}{\sqrt{n + 1}} \sin\left(\frac{4\pi}{n+1} \right) \\[5pt]
  \frac{2}{\sqrt{n + 1}} \sin\left(\frac{8\pi}{n+1} \right) \\[5pt]
  \vdots \\[5pt]
  \frac{2}{\sqrt{n + 1}} \sin\left(\frac{2n\pi}{n+1} \right)
\end{array}
\right]
\end{equation}
then
\begin{equation*}
\mathbf{H}_{n} = \mathbf{S}_{n} \mathbf{P}_{n} \left[
\begin{array}{cc}
\boldsymbol{\Phi}_{\frac{n}{2}}  & \mathbf{O} \\
\mathbf{O} & \boldsymbol{\Psi}_{\frac{n}{2}}
\end{array}
\right] \mathbf{P}_{n}^{\top} \mathbf{S}_{n}
\end{equation*}
where $\mathbf{S}_{n}$ is the $n \times n$ matrix \eqref{eq:2.1}, $\mathbf{P}_{n}$ is the $n \times n$ permutation matrix defined by
\begin{equation}\label{eq:2.4c}
\left[\mathbf{P}_{n} \right]_{k,\ell} = \left\{
\begin{array}{l}
1 \; \; \textit{if $k = 2\ell - 1$ or $k = 2\ell - n$} \\[5pt]
0 \; \; \textit{otherwise}
\end{array}
\right.
\end{equation}
and
\begin{gather}
\boldsymbol{\Phi}_{\frac{n}{2}} = \mathrm{diag} \left(\lambda_{1},\lambda_{3},\ldots,\lambda_{n-1} \right) + (c + \xi - a) \mathbf{x} \mathbf{x}^{\top} + (d + \eta - b) \mathbf{x} \mathbf{y}^{\top} + (d + \eta - b) \mathbf{y} \mathbf{x}^{\top} \label{eq:2.4d} \\
\boldsymbol{\Psi}_{\frac{n}{2}} = \mathrm{diag} \left(\lambda_{2},\lambda_{4},\ldots,\lambda_{n} \right) + (c + \xi - a) \mathbf{v} \mathbf{v}^{\top} + (d + \eta - b) \mathbf{v} \mathbf{w}^{\top} + (d + \eta - b) \mathbf{w} \mathbf{v}^{\top}. \label{eq:2.4e}
\end{gather}
\end{subequations}

\medskip

\noindent \textnormal{(b)} If $n$ is odd,
\begin{subequations}
\begin{equation}\label{eq:2.5a}
\mathbf{x} = \left[
\begin{array}{c}
  \frac{2}{\sqrt{n + 1}} \sin\left(\frac{\pi}{n+1} \right) \\[5pt]
  \frac{2}{\sqrt{n + 1}} \sin\left(\frac{3\pi}{n+1} \right) \\[5pt]
  \vdots \\[5pt]
  \frac{2}{\sqrt{n + 1}} \sin\left(\frac{n\pi}{n+1} \right)
\end{array}
\right], \quad
\mathbf{y} = \left[
\begin{array}{c}
  \frac{2}{\sqrt{n + 1}} \sin\left(\frac{2\pi}{n+1} \right) \\[5pt]
  \frac{2}{\sqrt{n + 1}} \sin\left(\frac{6\pi}{n+1} \right) \\[5pt]
  \vdots \\[5pt]
  \frac{2}{\sqrt{n + 1}} \sin\left(\frac{2n\pi}{n+1} \right)
\end{array}
\right]
\end{equation}
and
\begin{equation}\label{eq:2.5b}
\mathbf{v} = \left[
\begin{array}{c}
  \frac{2}{\sqrt{n + 1}} \sin\left(\frac{2\pi}{n+1} \right) \\[5pt]
  \frac{2}{\sqrt{n + 1}} \sin\left(\frac{4\pi}{n+1} \right) \\[5pt]
  \vdots \\[5pt]
  \frac{2}{\sqrt{n + 1}} \sin\left[\frac{(n - 1)\pi}{n+1} \right]
\end{array}
\right], \quad
\mathbf{w} = \left[
\begin{array}{c}
  \frac{2}{\sqrt{n + 1}} \sin\left(\frac{4\pi}{n+1} \right) \\[5pt]
  \frac{2}{\sqrt{n + 1}} \sin\left(\frac{8\pi}{n+1} \right) \\[5pt]
  \vdots \\[5pt]
  \frac{2}{\sqrt{n + 1}} \sin\left[\frac{2(n - 1)\pi}{n+1} \right]
\end{array}
\right]
\end{equation}
then
\begin{equation*}
\mathbf{H}_{n} = \mathbf{S}_{n} \mathbf{P}_{n} \left[
\begin{array}{cc}
\boldsymbol{\Phi}_{\frac{n+1}{2}}  & \mathbf{O} \\
\mathbf{O} & \boldsymbol{\Psi}_{\frac{n-1}{2}}
\end{array}
\right] \mathbf{P}_{n}^{\top} \mathbf{S}_{n}
\end{equation*}
where $\mathbf{S}_{n}$ is the $n \times n$ matrix \eqref{eq:2.1}, $\mathbf{P}_{n}$ is the $n \times n$ permutation matrix defined by
\begin{equation}\label{eq:2.5c}
\left[\mathbf{P}_{n} \right]_{k,\ell} = \left\{
\begin{array}{l}
1 \; \; \textit{if $k = 2\ell - 1$ or $k = 2\ell - n - 1$} \\[5pt]
0 \; \; \textit{otherwise}
\end{array}
\right.
\end{equation}
and
\begin{gather}
\boldsymbol{\Phi}_{\frac{n+1}{2}} = \mathrm{diag} \left(\lambda_{1},\lambda_{3},\ldots,\lambda_{n} \right) + (c + \xi - a) \mathbf{x} \mathbf{x}^{\top} + (d + \eta - b) \mathbf{x} \mathbf{y}^{\top} + (d + \eta - b) \mathbf{y} \mathbf{x}^{\top} \label{eq:2.5d} \\
\boldsymbol{\Psi}_{\frac{n-1}{2}} = \mathrm{diag} \left(\lambda_{2},\lambda_{4},\ldots,\lambda_{n-1} \right) + (c + \xi - a) \mathbf{v} \mathbf{v}^{\top} + (d + \eta - b) \mathbf{v} \mathbf{w}^{\top} + (d + \eta - b) \mathbf{w} \mathbf{v}^{\top}. \label{eq:2.5e}
\end{gather}
\end{subequations}
\end{lem}

\begin{proof}
Consider an integer $n \geqslant 5$, $a,b,c,d, \xi, \eta \in \mathbb{C}$, $\lambda_{k}$, $k = 1,\ldots,n$ given by \eqref{eq:2.3} and $\mathbf{H}_{n}$ the $n \times n$ matrix \eqref{eq:1.1}. Setting $\theta := c + \xi - a$, $\vartheta := d + \eta - b$,
\begin{equation*}
    \mathbf{\widehat{E}}_{n} = \left[
    \begin{array}{ccccc}
    c + \xi - a & 0 & \ldots & \ldots & 0 \\
    0 & 0 & \ddots & & \vdots \\
    \vdots & \ddots & \ddots & \ddots & \vdots \\
    \vdots & & \ddots & 0 & 0 \\
    0 & \ldots & \ldots & 0 & c + \xi - a
    \end{array}
    \right]
\end{equation*}
and
\begin{equation*}
    \mathbf{\widehat{F}}_{n} = \left[
    \begin{array}{cccccc}
    0 & d + \eta - b & 0 & \ldots & \ldots & 0 \\
    d + \eta - b & 0 & 0 & \ddots & & \vdots \\
    0 & 0 & \ddots & \ddots & \ddots & \vdots \\
    \vdots & \ddots & \ddots & \ddots & 0 & 0 \\
    \vdots & & \ddots & 0 & 0 & d + \eta - b \\
    0 & \ldots & \ldots & 0 & d + \eta - b & 0
    \end{array}
    \right]
\end{equation*}
we have, from Lemma~\ref{lem2},
\begin{equation*}
    \mathbf{S}_{n} \mathbf{H}_{n} \mathbf{S}_{n} = \mathbf{S}_{n} \left(\mathbf{\widehat{H}}_{n} + \mathbf{\widehat{E}}_{n} + \mathbf{\widehat{F}}_{n} \right) \mathbf{S}_{n} = \boldsymbol{\Lambda}_{n} + \mathbf{G}_{n} + \mathbf{K}_{n}
\end{equation*}
where $\mathbf{S}_{n}$ is the $n \times n$ matrix \eqref{eq:2.1}, $\mathbf{\widehat{H}}_{n}$ is the $n \times n$ matrix \eqref{eq:2.2},
\begin{gather*}
    \boldsymbol{\Lambda}_{n} = \mathrm{diag} \left(\lambda_{1},\ldots,\lambda_{n} \right), \\
    \left[\mathbf{G}_{n} \right]_{k,\ell} = \frac{2 \theta}{n+1}\sin \left(\frac{k\pi}{n+1} \right) \sin \left(\frac{\ell\pi}{n+1} \right) \left[1 + (-1)^{k + \ell} \right]
\end{gather*}
and
\begin{equation*}
    \left[\mathbf{K}_{n} \right]_{k,\ell} = \frac{2\vartheta}{n+1} \left[\sin \left(\frac{k\pi}{n+1} \right) \sin \left(\frac{2\ell\pi}{n+1} \right) + \sin \left(\frac{2k\pi}{n+1} \right) \sin \left(\frac{\ell\pi}{n+1} \right) \right] \left[1 + (-1)^{k + \ell} \right].
\end{equation*}
Since $\left[\mathbf{G}_{n} \right]_{k,\ell} = \left[\mathbf{K}_{n} \right]_{k,\ell} = 0$ whenever $k + \ell$ is odd, we can permute rows and columns of $\boldsymbol{\Lambda}_{n} + \mathbf{G}_{n} + \mathbf{K}_{n}$ according to the permutation matrices \eqref{eq:2.4c} and \eqref{eq:2.5c}, yielding: for $n$ even,
\begin{equation*}
    \mathbf{H}_{n} = \mathbf{S}_{n} \mathbf{P}_{n} \left[
    \begin{array}{cc}
    \boldsymbol{\Upsilon}_{\frac{n}{2}} + \theta \mathbf{x} \mathbf{x}^{\top} + \vartheta \mathbf{x} \mathbf{y}^{\top} + \vartheta \mathbf{y} \mathbf{x}^{\top} & \mathbf{O} \\
    \mathbf{O} & \boldsymbol{\Delta}_{\frac{n}{2}} + \theta \mathbf{v} \mathbf{v}^{\top} + \vartheta \mathbf{v} \mathbf{w}^{\top} + \vartheta \mathbf{w} \mathbf{v}^{\top}
    \end{array}
    \right] \mathbf{P}_{n}^{\top} \mathbf{S}_{n},
\end{equation*}
where $\mathbf{P}_{n}$ is the matrix \eqref{eq:2.4c}, $\boldsymbol{\Upsilon}_{\frac{n}{2}} = \mathrm{diag}(\lambda_{1},\lambda_{3},\ldots,\lambda_{n-1})$, $\boldsymbol{\Delta}_{\frac{n}{2}} = \mathrm{diag}(\lambda_{2},\lambda_{4},\ldots,\lambda_{n})$ and $\mathbf{x}$, $\mathbf{y}$ are given by \eqref{eq:2.4a}; for $n$ odd,
\begin{equation*}
    \mathbf{H}_{n} = \mathbf{S}_{n} \mathbf{P}_{n} \left[
    \begin{array}{cc}
    \boldsymbol{\Upsilon}_{\frac{n+1}{2}} + \theta \mathbf{x} \mathbf{x}^{\top} + \vartheta \mathbf{x} \mathbf{y}^{\top} + \vartheta \mathbf{y} \mathbf{x}^{\top} & \mathbf{O} \\
    \mathbf{O} & \boldsymbol{\Delta}_{\frac{n-1}{2}} + \theta \mathbf{v} \mathbf{v}^{\top} + \vartheta \mathbf{v} \mathbf{w}^{\top} + \vartheta \mathbf{w} \mathbf{v}^{\top}
    \end{array}
    \right] \mathbf{P}_{n}^{\top} \mathbf{S}_{n},
\end{equation*}
with $\mathbf{P}_{n}$ defined in \eqref{eq:2.5c}, $\boldsymbol{\Upsilon}_{\frac{n+1}{2}} = \mathrm{diag}(\lambda_{1},\lambda_{3},\ldots,\lambda_{n})$, $\boldsymbol{\Delta}_{\frac{n-1}{2}} = \mathrm{diag}(\lambda_{2},\lambda_{4},\ldots,\lambda_{n-1})$ and $\mathbf{v}$, $\mathbf{w}$ defined by \eqref{eq:2.5a}. The proof is completed.
\end{proof}

\begin{rem}
Let us point out that the decomposition for real heptadiagonal symmetric Toeplitz matrices established in Proposition 3.1 of \cite{Bini83} at the expense of the bordering technique is no more useful for matrices having the shape \eqref{eq:1.1}. As consequence, some results stated by these authors will be necessarily extended, particularly, the referred decomposition and a formula to compute the determinant of real heptadiagonal symmetric Toeplitz matrices (Corollary 3.1 of \cite{Bini83}).
\end{rem}

\bigskip

\section{Main results}

\subsection{Determinant of $\mathbf{H}_{n}$}

\indent

The orthogonal block diagonalization established in Lemma~\ref{lem3} will lead us to an explicit formula for the determinant of the matrix $\mathbf{H}_{n}$.


\begin{thr}\label{thr1}
Let $n \geqslant 5$ be an integer, $a,b,c,d, \xi, \eta \in \mathbb{C}$, $\lambda_{k}$, $k = 1,\ldots,n$ be given by \eqref{eq:2.3}, $x_{k} = \sin \left(\frac{k \pi}{n+1} \right)$, $k=1,\ldots,2n$ and $\mathbf{H}_{n}$ the $n \times n$ matrix \eqref{eq:1.1}. If $\theta := c + \xi - a$, $\vartheta := d + \eta - b$ and \\

\noindent \textnormal{(a)} $n$ is even then
\begin{equation*}
\begin{split}
\det & (\mathbf{H}_{n}) = \Bigg[\prod_{j=1}^{\frac{n}{2}} \lambda_{2j} + \sum_{k=1}^{\frac{n}{2}} \prod_{\substack{j=1 \\ j \neq k}}^{\frac{n}{2}} \lambda_{2j} \tfrac{4 \theta x_{2k}^{2} + 8 \vartheta x_{2k} x_{4k}}{(n+1)} - \sum_{1 \leqslant k < \ell \leqslant \frac{n}{2}} \prod_{\substack{j=1 \\ j \neq k, \ell}}^{\frac{n}{2}} \lambda_{2j} \tfrac{16\vartheta^{2} \left(x_{2k} x_{4 \ell} - x_{2 \ell} x_{4k} \right)^{2}}{(n+1)^{2}} \Bigg] \times \\
&\Bigg[\prod_{j=1}^{\frac{n}{2}} \lambda_{2j-1} + \sum_{k=1}^{\frac{n}{2}} \prod_{\substack{j=1 \\ j \neq k}}^{\frac{n}{2}} \lambda_{2j-1} \tfrac{4 \theta x_{2k-1}^{2} + 8 \vartheta x_{2k-1} x_{4k-2}}{(n+1)} - \sum_{1 \leqslant k < \ell \leqslant \frac{n}{2}} \prod_{\substack{j=1 \\ j \neq k, \ell}}^{\frac{n}{2}} \lambda_{2j-1} \tfrac{16\vartheta^{2} \left(x_{2k-1} x_{4 \ell - 2} - x_{2 \ell - 1} x_{4k-2} \right)^{2}}{(n+1)^{2}} \Bigg].
\end{split}
\end{equation*}

\medskip

\noindent \textnormal{(b)} $n$ is odd then
\begin{equation*}
\begin{split}
& \det (\mathbf{H}_{n}) = \Bigg[\prod_{j=1}^{\frac{n-1}{2}} \lambda_{2j} + \sum_{k=1}^{\frac{n-1}{2}} \prod_{\substack{j=1 \\ j \neq k}}^{\frac{n-1}{2}} \lambda_{2j} \tfrac{4 \theta x_{2k}^{2} + 8 \vartheta x_{2k} x_{4k}}{(n+1)} - \sum_{1 \leqslant k < \ell \leqslant \frac{n-1}{2}} \prod_{\substack{j=1 \\ j \neq k, \ell}}^{\frac{n-1}{2}} \lambda_{2j} \tfrac{16\vartheta^{2} \left(x_{2k} x_{4 \ell} - x_{2 \ell} x_{4k} \right)^{2}}{(n+1)^{2}} \Bigg] \times \\
&\Bigg[\prod_{j=1}^{\frac{n+1}{2}} \lambda_{2j-1} + \sum_{k=1}^{\frac{n+1}{2}} \prod_{\substack{j=1 \\ j \neq k}}^{\frac{n+1}{2}} \lambda_{2j-1} \tfrac{4 \theta x_{2k-1}^{2} + 8 \vartheta x_{2k-1} x_{4k-2}}{(n+1)} - \sum_{1 \leqslant k < \ell \leqslant \frac{n+1}{2}} \prod_{\substack{j=1 \\ j \neq k, \ell}}^{\frac{n+1}{2}} \lambda_{2j-1} \tfrac{16\vartheta^{2} \left(x_{2k-1} x_{4 \ell - 2} - x_{2 \ell - 1} x_{4k-2} \right)^{2}}{(n+1)^{2}} \Bigg].
\end{split}
\end{equation*}
\end{thr}

\begin{proof}
Since both assertions can be proven in the same way, we only prove (a). Consider an even integer $n \geqslant 5$, $a,b,c,d, \xi, \eta \in \mathbb{C}$, $x_{k} = \sin \left(\frac{k \pi}{n+1} \right)$, $k=1,\ldots,2n$, $\lambda_{k}$, $k = 1,\ldots,n$ given by \eqref{eq:2.3}, $\theta := c + \xi - a$, $\vartheta := d + \eta - b$ and the notations used in Lemma~\ref{lem3}. The determinant formula for block-triangular matrices (see \cite{Harville97}, page $185$) and Lemma~\ref{lem3} ensure $\det(\mathbf{H}_{n}) = \det \left(\boldsymbol{\Phi}_{\frac{n}{2}} \right) \det \left(\boldsymbol{\Psi}_{\frac{n}{2}} \right)$. We shall first assume $\lambda_{k} \neq 0$ for all $k=1,\ldots,n$,
\begin{subequations}
\begin{gather}
\frac{4\theta}{n+1} \sum_{k=1}^{\frac{n}{2}} \frac{x_{2k-1}^{2}}{\lambda_{2k-1}} \neq -1 \label{eq:3.1a} \\
\frac{4\theta}{n+1} \sum_{k=1}^{\frac{n}{2}} \frac{x_{2k-1}^{2}}{\lambda_{2k-1}} + \frac{4\vartheta}{n+1} \sum_{k=1}^{\frac{n}{2}} \frac{x_{2k-1}x_{4k-2}}{\lambda_{2k-1}} \neq -1 \label{eq:3.1b} \\
\sum_{k=1}^{\frac{n}{2}} \frac{4\theta x_{2k-1}^{2} + 8\vartheta x_{2k-1} x_{4k-2}}{(n+1) \lambda_{2k-1}} - \frac{16\vartheta^{2}}{(n+1)^{2}} \sum_{1 \leqslant k < \ell \leqslant \frac{n}{2}} \frac{\left(x_{2k-1} x_{4 \ell - 2} - x_{2 \ell - 1} x_{4k-2} \right)^{2}}{\lambda_{2k-1} \lambda_{2 \ell - 1}} \neq -1 \label{eq:3.1c}
\end{gather}
\end{subequations}
and
\begin{subequations}
\begin{gather}
\frac{4\theta}{n+1} \sum_{k=1}^{\frac{n}{2}} \frac{x_{2k}^{2}}{\lambda_{2k}} \neq -1 \label{eq:3.2a} \\
\frac{4\theta}{n+1} \sum_{k=1}^{\frac{n}{2}} \frac{x_{2k}^{2}}{\lambda_{2k}} + \frac{4\vartheta}{n+1} \sum_{k=1}^{\frac{n}{2}} \frac{x_{2k}x_{4k}}{\lambda_{2k}} \neq -1 \label{eq:3.2b} \\
\sum_{k=1}^{\frac{n}{2}} \frac{4\theta x_{2k}^{2} + 8\vartheta x_{2k} x_{4k}}{(n+1) \lambda_{2k}} - \frac{16\vartheta^{2}}{(n+1)^{2}} \sum_{1 \leqslant k < \ell \leqslant \frac{n}{2}} \frac{\left(x_{2k} x_{4 \ell} - x_{2 \ell} x_{4k} \right)^{2}}{\lambda_{2k} \lambda_{2 \ell}} \neq -1. \label{eq:3.2c}
\end{gather}
\end{subequations}
Putting $\boldsymbol{\Upsilon}_{\frac{n}{2}} := \mathrm{diag} \left(\lambda_{1},\lambda_{3},\ldots,\lambda_{n-1} \right)$ and $\boldsymbol{\Delta}_{\frac{n}{2}} := \mathrm{diag} \left(\lambda_{2},\lambda_{4},\ldots,\lambda_{n} \right)$, we have
\begin{align*}
\det \left(\boldsymbol{\Phi}_{\frac{n}{2}} \right) &= \det \left(\boldsymbol{\Upsilon}_{\frac{n}{2}} + \theta \mathbf{x} \mathbf{x}^{\top} + \vartheta \mathbf{x} \mathbf{y}^{\top} + \vartheta \mathbf{y} \mathbf{x}^{\top} \right) \\
& = \bigg[1 + \theta \mathbf{x}^{\top} \boldsymbol{\Upsilon}_{\frac{n}{2}}^{-1} \mathbf{x} + 2 \vartheta \mathbf{x}^{\top} \boldsymbol{\Upsilon}_{\frac{n}{2}}^{-1} \mathbf{y} + \vartheta^{2} \left(\mathbf{x}^{\top} \boldsymbol{\Upsilon}_{\frac{n}{2}}^{-1} \mathbf{y} \right)^{2} - \vartheta^{2} \left(\mathbf{x}^{\top} \boldsymbol{\Upsilon}_{\frac{n}{2}}^{-1} \mathbf{x} \right) \left(\mathbf{y}^{\top} \boldsymbol{\Upsilon}_{\frac{n}{2}}^{-1} \mathbf{y} \right) \bigg] \det \left(\boldsymbol{\Upsilon}_{\frac{n}{2}} \right) \\
& = \Bigg[\prod_{j=1}^{\frac{n}{2}} \lambda_{2j-1} + \sum_{k=1}^{\frac{n}{2}} \prod_{\substack{j=1 \\ j \neq k}}^{\frac{n}{2}} \lambda_{2j-1} \tfrac{4 \theta x_{2k-1}^{2} + 8 \vartheta x_{2k-1} x_{4k-2}}{(n+1)} + \Bigg. \\
& \hspace*{6.2cm} \Bigg. - \sum_{1 \leqslant k < \ell \leqslant \frac{n}{2}} \prod_{\substack{j=1 \\ j \neq k, \ell}}^{\frac{n}{2}} \lambda_{2j-1} \tfrac{16\vartheta^{2} \left(x_{2k-1} x_{4 \ell - 2} - x_{2 \ell - 1} x_{4k-2} \right)^{2}}{(n+1)^{2}} \Bigg]
\end{align*}
and
\begin{align*}
\det & \left(\boldsymbol{\Psi}_{\frac{n}{2}} \right) = \\
& = \det \left(\boldsymbol{\Delta}_{\frac{n}{2}} + \theta \mathbf{v} \mathbf{v}^{\top} + \vartheta \mathbf{v} \mathbf{w}^{\top} + \vartheta \mathbf{w} \mathbf{v}^{\top} \right) \\
& = \bigg[1 + \theta \mathbf{v}^{\top} \boldsymbol{\Delta}_{\frac{n}{2}}^{-1} \mathbf{v} + 2 \vartheta \mathbf{v}^{\top} \boldsymbol{\Delta}_{\frac{n}{2}}^{-1} \mathbf{w} + \vartheta^{2} \left(\mathbf{v}^{\top} \boldsymbol{\Delta}_{\frac{n}{2}}^{-1} \mathbf{w} \right)^{2} - \vartheta^{2} \left(\mathbf{v}^{\top} \boldsymbol{\Delta}_{\frac{n}{2}}^{-1} \mathbf{v} \right) \left(\mathbf{w}^{\top} \boldsymbol{\Delta}_{\frac{n}{2}}^{-1} \mathbf{w} \right) \bigg] \det \left(\boldsymbol{\Delta}_{\frac{n}{2}} \right) \\
& = \Bigg[\prod_{j=1}^{\frac{n}{2}} \lambda_{2j} + \sum_{k=1}^{\frac{n}{2}} \prod_{\substack{j=1 \\ j \neq k}}^{\frac{n}{2}} \lambda_{2j} \tfrac{4 \theta x_{2k}^{2} + 8 \vartheta x_{2k} x_{4k}}{(n+1)} - \sum_{1 \leqslant k < \ell \leqslant \frac{n}{2}} \prod_{\substack{j=1 \\ j \neq k, \ell}}^{\frac{n}{2}} \lambda_{2j} \tfrac{16\vartheta^{2} \left(x_{2k} x_{4 \ell} - x_{2 \ell} x_{4k} \right)^{2}}{(n+1)^{2}} \Bigg]
\end{align*}
(see \cite{Miller81}, page $69$ and $70$), i.e.
\begin{equation}\label{eq:3.3}
\begin{split}
& \det (\mathbf{H}_{n}) = \Bigg[\prod_{j=1}^{\frac{n}{2}} \lambda_{2j} + \sum_{k=1}^{\frac{n}{2}} \prod_{\substack{j=1 \\ j \neq k}}^{\frac{n}{2}} \lambda_{2j} \tfrac{4 \theta x_{2k}^{2} + 8 \vartheta x_{2k} x_{4k}}{(n+1)} - \sum_{1 \leqslant k < \ell \leqslant \frac{n}{2}} \prod_{\substack{j=1 \\ j \neq k, \ell}}^{\frac{n}{2}} \lambda_{2j} \tfrac{16\vartheta^{2} \left(x_{2k} x_{4 \ell} - x_{2 \ell} x_{4k} \right)^{2}}{(n+1)^{2}} \Bigg] \cdot \\
&\Bigg[\prod_{j=1}^{\frac{n}{2}} \lambda_{2j-1} + \sum_{k=1}^{\frac{n}{2}} \prod_{\substack{j=1 \\ j \neq k}}^{\frac{n}{2}} \lambda_{2j-1} \tfrac{4 \theta x_{2k-1}^{2} + 8 \vartheta x_{2k-1} x_{4k-2}}{(n+1)} - \sum_{1 \leqslant k < \ell \leqslant \frac{n}{2}} \prod_{\substack{j=1 \\ j \neq k, \ell}}^{\frac{n}{2}} \lambda_{2j-1} \tfrac{16\vartheta^{2} \left(x_{2k-1} x_{4 \ell - 2} - x_{2 \ell - 1} x_{4k-2} \right)^{2}}{(n+1)^{2}} \Bigg].
\end{split}
\end{equation}
Since both sides of \eqref{eq:3.3} are polynomials in the variables $a,b,c,d, \xi, \eta$, conditions \eqref{eq:3.1a}, \eqref{eq:3.1b}, \eqref{eq:3.1c}, \eqref{eq:3.2a}, \eqref{eq:3.2b}, \eqref{eq:3.2c} as well as $\lambda_{k} \neq 0$ can be dropped and \eqref{eq:3.3} is valid more generally.
\end{proof}

\subsection{Eigenvalue localization for $\mathbf{H}_{n}$}

\indent

The next lemma will allow us to express the eigenvalues of key matrices in this paper as the zeros of explicit rational functions providing, additionally, explicit upper and lower bounds for each one. Throughout, $\norm{\; \cdot \; }$ will denote the Euclidean norm.


\begin{lem}\label{lem4}
Let $n \geqslant 5$ be an integer, $a,b,c,d, \xi, \eta \in \mathbb{R}$ and $\lambda_{k}$, $k = 1,\ldots,n$ be given by \eqref{eq:2.3}.

\medskip

\noindent \textnormal{(a)} If $n$ is even,

\begin{itemize}
\begin{subequations}
\item[\textnormal{i.}] $\mathbf{x}, \mathbf{y}$ are given by \eqref{eq:2.4a} and the eigenvalues of
    \begin{equation}\label{eq:3.4a}
    \mathrm{diag} \left(\lambda_{1},\lambda_{3},\ldots,\lambda_{n-1} \right) + (c + \xi - a) \mathbf{x} \mathbf{x}^{\top} + (d + \eta - b) \mathbf{x} \mathbf{y}^{\top} + (d + \eta - b) \mathbf{y} \mathbf{x}^{\top}
    \end{equation}
    are not of the form $a + 2b \cos \left[\frac{(2k-1)\pi}{n+1} \right] + 2c \cos \left[\frac{2(2k-1) \pi}{n+1} \right] + 2d \cos \left[\frac{3(2k-1) \pi}{n+1} \right]$, $k = 1,\ldots,\frac{n}{2}$ then the eigenvalues of \eqref{eq:3.4a} are precisely the zeros of the rational function
    \begin{equation}\label{eq:3.4b}
    \begin{split}
    f(t) &= 1 + \frac{4}{n+1} \sum_{k=1}^{\frac{n}{2}} \frac{(c + \xi - a) \sin^{2} \left[\frac{(2k - 1)\pi}{n+1} \right] + 2(d + \eta - b)\sin \left[\frac{(2k - 1)\pi}{n+1} \right] \sin \left[\frac{(4k - 2)\pi}{n+1} \right]}{\lambda_{2k - 1} - t} + \\
    & - \frac{16(d + \eta - b)^{2}}{(n+1)^{2}} \sum_{1 \leqslant k < \ell \leqslant \frac{n}{2}} \frac{\left\{\sin \left[\frac{(2k - 1)\pi}{n+1} \right] \sin \left[\frac{(4 \ell - 2)\pi}{n+1} \right] - \sin \left[\frac{(4k - 2)\pi}{n+1} \right] \sin \left[\frac{(2 \ell - 1)\pi}{n+1} \right] \right\}^{2}}{(\lambda_{2k - 1} - t)(\lambda_{2\ell - 1} - t)}.
    \end{split}
    \end{equation}
    Moreover, if $\mu_{1} \leqslant \mu_{2} \leqslant \ldots \leqslant \mu_{\frac{n}{2}}$ are the eigenvalues of \eqref{eq:3.4a} and $\lambda_{\tau(1)} \leqslant \lambda_{\tau(3)} \leqslant \ldots \leqslant \lambda_{\tau(n-1)}$ are arranged in non-decreasing order by some bijection $\tau$ defined in $\{1,3, \ldots, n-1\}$ then
    \begin{equation}\label{eq:3.4c}
    \lambda_{\tau(2k-1)} + \tfrac{(c + \xi - a) - \sqrt{(c + \xi - a)^{2} + 4(d + \eta - b)^{2}}}{2} \leqslant \mu_{k} \leqslant \lambda_{\tau(2k-1)} + \tfrac{(c + \xi - a) + \sqrt{(c + \xi - a)^{2} + 4(d + \eta - b)^{2}}}{2}
    \end{equation}
    for each $k=1,\ldots,\tfrac{n}{2}$.
\end{subequations}

\medskip

\begin{subequations}
\item[\textnormal{ii.}] $\mathbf{v}, \mathbf{w}$ are given by \eqref{eq:2.4b} and the eigenvalues of
    \begin{equation}\label{eq:3.5a}
    \mathrm{diag} \left(\lambda_{2},\lambda_{4},\ldots,\lambda_{n} \right) + (c + \xi - a) \mathbf{v} \mathbf{v}^{\top} + (d + \eta - b) \mathbf{v} \mathbf{w}^{\top} + (d + \eta - b) \mathbf{w} \mathbf{v}^{\top}
    \end{equation}
    are not of the form $a + 2b \cos \left(\frac{2k\pi}{n+1} \right) + 2c \cos \left(\frac{4k \pi}{n+1} \right) + 2d \cos \left(\frac{6k \pi}{n+1} \right)$, $k = 1,\ldots,\frac{n}{2}$ then the eigenvalues of \eqref{eq:3.5a} are precisely the zeros of the rational function
    \begin{equation}\label{eq:3.5b}
    \begin{split}
    g(t) = 1 & + \frac{4}{n+1} \sum_{k=1}^{\frac{n}{2}} \frac{(c + \xi - a) \sin^{2} \left(\frac{2k\pi}{n+1} \right) + 2(d + \eta - b)\sin \left(\frac{2k\pi}{n+1} \right) \sin \left(\frac{4k\pi}{n+1} \right)}{\lambda_{2k} - t} + \\
    & - \frac{16(d + \eta - b)^{2}}{(n+1)^{2}} \sum_{1 \leqslant k < \ell \leqslant \frac{n}{2}} \frac{\left[\sin \left(\frac{2k\pi}{n+1} \right) \sin \left(\frac{4 \ell\pi}{n+1} \right) - \sin \left(\frac{4k\pi}{n+1} \right) \sin \left(\frac{2 \ell\pi}{n+1} \right) \right]^{2}}{(\lambda_{2k} - t)(\lambda_{2\ell} - t)}.
    \end{split}
    \end{equation}
    Furthermore, if $\nu_{1} \leqslant \nu_{2} \leqslant \ldots \leqslant \nu_{\frac{n}{2}}$ are the eigenvalues of \eqref{eq:3.5a} and $\lambda_{\sigma(2)} \leqslant \lambda_{\sigma(4)} \leqslant \ldots \leqslant \lambda_{\sigma(n)}$ are arranged in non-decreasing order by some bijection $\sigma$ defined in $\{2,4, \ldots, n\}$ then
    \begin{equation}\label{eq:3.5c}
    \lambda_{\sigma(2k)} + \tfrac{(c + \xi - a) - \sqrt{(c + \xi - a)^{2} + 4(d + \eta - b)^{2}}}{2} \leqslant \nu_{k} \leqslant \lambda_{\sigma(2k)} + \tfrac{(c + \xi - a) + \sqrt{(c + \xi - a)^{2} + 4(d + \eta - b)^{2}}}{2}
    \end{equation}
    for every $k=1,\ldots,\tfrac{n}{2}$.
\end{subequations}
\end{itemize}

\medskip

\noindent \textnormal{(b)} If $n$ is odd,

\begin{itemize}
\begin{subequations}
\item[\textnormal{i.}] $\mathbf{x}, \mathbf{y}$ are given by \eqref{eq:2.5a} and the eigenvalues of
    \begin{equation}\label{eq:3.6a}
    \mathrm{diag} \left(\lambda_{1},\lambda_{3},\ldots,\lambda_{n} \right) + (c + \xi - a) \mathbf{x} \mathbf{x}^{\top} + (d + \eta - b) \mathbf{x} \mathbf{y}^{\top} + (d + \eta - b) \mathbf{y} \mathbf{x}^{\top}
    \end{equation}
    are not of the form $a + 2b \cos \left[\frac{(2k-1)\pi}{n+1} \right] + 2c \cos \left[\frac{2(2k-1) \pi}{n+1} \right] + 2d \cos \left[\frac{3(2k-1) \pi}{n+1} \right]$, $k = 1,\ldots,\frac{n+1}{2}$ then the eigenvalues of \eqref{eq:3.6a} are precisely the zeros of the rational function
    \begin{equation}\label{eq:3.6b}
    \begin{split}
    f(t) &= 1 + \frac{4}{n+1} \sum_{k=1}^{\frac{n+1}{2}} \frac{(c + \xi - a) \sin^{2} \left[\frac{(2k - 1)\pi}{n+1} \right] + 2(d + \eta - b)\sin \left[\frac{(2k - 1)\pi}{n+1} \right] \sin \left[\frac{(4k - 2)\pi}{n+1} \right]}{\lambda_{2k - 1} - t} + \\
    & - \frac{16(d + \eta - b)^{2}}{(n+1)^{2}} \sum_{1 \leqslant k < \ell \leqslant \frac{n+1}{2}} \frac{\left\{\sin \left[\frac{(2k - 1)\pi}{n+1} \right] \sin \left[\frac{(4 \ell - 2)\pi}{n+1} \right] - \sin \left[\frac{(4k - 2)\pi}{n+1} \right] \sin \left[\frac{(2 \ell - 1)\pi}{n+1} \right] \right\}^{2}}{(\lambda_{2k - 1} - t)(\lambda_{2\ell - 1} - t)}.
    \end{split}
    \end{equation}
    Moreover, if $\mu_{1} \leqslant \mu_{2} \leqslant \ldots \leqslant \mu_{\frac{n+1}{2}}$ are the eigenvalues of \eqref{eq:3.6a} and $\lambda_{\tau(1)} \leqslant \lambda_{\tau(3)} \leqslant \ldots \leqslant \lambda_{\tau(n)}$ are arranged in non-decreasing order by some bijection $\tau$ defined in $\{1,3, \ldots, n\}$ then
    \begin{equation}\label{eq:3.6c}
    \lambda_{\tau(2k-1)} + \tfrac{(c + \xi - a) - \sqrt{(c + \xi - a)^{2} + 4(d + \eta - b)^{2}}}{2} \leqslant \mu_{k} \leqslant \lambda_{\tau(2k-1)} + \tfrac{(c + \xi - a) + \sqrt{(c + \xi - a)^{2} + 4(d + \eta - b)^{2}}}{2}
    \end{equation}
    for any $k=1,\ldots,\tfrac{n+1}{2}$.
\end{subequations}

\medskip

\begin{subequations}
\item[\textnormal{ii.}] $\mathbf{v}, \mathbf{w}$ are given by \eqref{eq:2.5b} and the eigenvalues of
    \begin{equation}\label{eq:3.7a}
    \mathrm{diag} \left(\lambda_{2},\lambda_{4},\ldots,\lambda_{n-1} \right) + (c + \xi - a) \mathbf{v} \mathbf{v}^{\top} + (d + \eta - b) \mathbf{v} \mathbf{w}^{\top} + (d + \eta - b) \mathbf{w} \mathbf{v}^{\top}
    \end{equation}
    are not of the form $a + 2b \cos \left(\frac{2k\pi}{n+1} \right) + 2c \cos \left(\frac{4k \pi}{n+1} \right) + 2d \cos \left(\frac{6k \pi}{n+1} \right)$, $k = 1,\ldots,\frac{n-1}{2}$ then the eigenvalues of \eqref{eq:3.7a} are precisely the zeros of the rational function
    \begin{equation}\label{eq:3.7b}
    \begin{split}
    g(t) = 1 & + \frac{4}{n+1} \sum_{k=1}^{\frac{n-1}{2}} \frac{(c + \xi - a) \sin^{2} \left(\frac{2k\pi}{n+1} \right) + 2(d + \eta - b)\sin \left(\frac{2k\pi}{n+1} \right) \sin \left(\frac{4k\pi}{n+1} \right)}{\lambda_{2k} - t} + \\
    & - \frac{16(d + \eta - b)^{2}}{(n+1)^{2}} \sum_{1 \leqslant k < \ell \leqslant \frac{n-1}{2}} \frac{\left[\sin \left(\frac{2k\pi}{n+1} \right) \sin \left(\frac{4 \ell\pi}{n+1} \right) - \sin \left(\frac{4k\pi}{n+1} \right) \sin \left(\frac{2 \ell\pi}{n+1} \right) \right]^{2}}{(\lambda_{2k} - t)(\lambda_{2\ell} - t)}.
    \end{split}
    \end{equation}
    Furthermore, if $\nu_{1} \leqslant \nu_{2} \leqslant \ldots \leqslant \nu_{\frac{n-1}{2}}$ are the eigenvalues of \eqref{eq:3.7a} and $\lambda_{\sigma(2)} \leqslant \lambda_{\sigma(4)} \leqslant \ldots \leqslant \lambda_{\sigma(n-1)}$ are arranged in non-decreasing order by some bijection $\sigma$ defined in $\{2,4, \ldots, n-1\}$ then
    \begin{equation}\label{eq:3.7c}
    \lambda_{\sigma(2k)} + \tfrac{(c + \xi - a) - \sqrt{(c + \xi - a)^{2} + 4(d + \eta - b)^{2}}}{2} \leqslant \nu_{k} \leqslant \lambda_{\sigma(2k)} + \tfrac{(c + \xi - a) + \sqrt{(c + \xi - a)^{2} + 4(d + \eta - b)^{2}}}{2}
    \end{equation}
    for all $k=1,\ldots,\tfrac{n-1}{2}$.
\end{subequations}
\end{itemize}
\end{lem}

\begin{proof}
Suppose $n \geqslant 5$ an even integer, $a,b,c,d, \xi, \eta \in \mathbb{R}$, $\lambda_{k}$, $k = 1,\ldots,n$ given by \eqref{eq:2.3} and put $\theta := c + \xi - a$, $\vartheta := d + \eta - b$. We shall denote by $\EuScript{S}(k,m)$ the collection of all $k$-element subsets of $\{1,2,\ldots,m \}$ written in increasing order; additionally, for any rectangular matrix $\mathbf{M}$, we shall indicate by $\det \left(\mathbf{M}[I,J] \right)$ the minor determined by the subsets $I = \left\{i_{1} < i_{2} < \ldots < i_{k} \right\}$ and $J = \left\{j_{1} < j_{2} < \ldots < j_{k} \right\}$. Supposing $\theta \neq 0$,
\begin{equation*}
\mathbf{X} = \left[
\begin{array}{cccc}
  2\sqrt{\frac{\theta}{n+1}} \sin \left(\frac{\pi}{n+1} \right) & 2\sqrt{\frac{\theta}{n+1}} \sin \left(\frac{3\pi}{n+1} \right) & \ldots & 2\sqrt{\frac{\theta}{n+1}} \sin \left(\frac{n\pi}{n+1} \right) \\
  2\sqrt{\frac{\theta}{n+1}} \sin \left(\frac{\pi}{n+1} \right) & 2\sqrt{\frac{\theta}{n+1}} \sin \left(\frac{3\pi}{n+1} \right) & \ldots & 2\sqrt{\frac{\theta}{n+1}} \sin \left(\frac{n\pi}{n+1} \right) \\
  \frac{2\vartheta}{\sqrt{\theta (n+1)}} \sin \left(\frac{2\pi}{n+1} \right) & \frac{2\vartheta}{\sqrt{\theta (n+1)}} \sin \left(\frac{6\pi}{n+1} \right) & \ldots & \frac{2\vartheta}{\sqrt{\theta (n+1)}} \sin \left[\frac{(4n-2)\pi}{n+1} \right]
\end{array}
\right]
\end{equation*}
and
\begin{equation*}
\mathbf{Y} = \left[
\begin{array}{cccc}
  2\sqrt{\frac{\theta}{n+1}} \sin \left(\frac{\pi}{n+1} \right) & 2\sqrt{\frac{\theta}{n+1}} \sin \left(\frac{3\pi}{n+1} \right) & \ldots & 2\sqrt{\frac{\theta}{n+1}} \sin \left(\frac{n\pi}{n+1} \right) \\
  \frac{2\vartheta}{\sqrt{\theta (n+1)}} \sin \left(\frac{2\pi}{n+1} \right) & \frac{2\vartheta}{\sqrt{\theta (n+1)}} \sin \left(\frac{6\pi}{n+1} \right) & \ldots & \frac{2\vartheta}{\sqrt{\theta (n+1)}} \sin \left[\frac{(4n-2)\pi}{n+1} \right] \\
  2\sqrt{\frac{\theta}{n+1}} \sin \left(\frac{\pi}{n+1} \right) & 2\sqrt{\frac{\theta}{n+1}} \sin \left(\frac{3\pi}{n+1} \right) & \ldots & 2\sqrt{\frac{\theta}{n+1}} \sin \left(\frac{n\pi}{n+1} \right)
\end{array}
\right],
\end{equation*}
Theorem 1 of \cite{Anderson96} ensures that $\zeta$ is an eigenvalue of \eqref{eq:3.4a} if and only if
\begin{equation*}
1 + \sum_{k=1}^{\frac{n}{2}} \sum_{J \in \EuScript{S}\left(k,\frac{n}{2} \right)} \sum_{I \in \EuScript{S}(k,3)} \frac{\det\left(\mathbf{X}[I,J] \right) \det\left(\mathbf{Y}[I,J] \right)}{\prod_{j \in J}(\lambda_{2j-1} - \zeta)} = 0
\end{equation*}
provided that $\zeta$ is not an eigenvalue of $\mathrm{diag} \left(\lambda_{1},\lambda_{3},\ldots,\lambda_{n-1} \right)$. Since
\begin{equation*}
\begin{split}
1 + \sum_{k=1}^{\frac{n}{2}} & \sum_{J \in \EuScript{S}\left(k,\frac{n}{2} \right)} \sum_{I \in \EuScript{S}(k,3)} \frac{\det\left(\mathbf{X}[I,J] \right) \det\left(\mathbf{Y}[I,J] \right)}{\prod_{j \in J}(\lambda_{2j-1} - \zeta)} = \\
&= 1 + \frac{4}{n+1} \sum_{k=1}^{\frac{n}{2}} \frac{\theta \sin^{2} \left[\frac{(2k - 1)\pi}{n+1} \right] + 2\vartheta\sin \left[\frac{(2k - 1)\pi}{n+1} \right] \sin \left[\frac{(4k - 2)\pi}{n+1} \right]}{\lambda_{2k - 1} - \zeta} + \\
& - \frac{16\vartheta^{2}}{(n+1)^{2}} \sum_{1 \leqslant k < \ell \leqslant \frac{n}{2}} \frac{\left\{\sin \left[\frac{(2k - 1)\pi}{n+1} \right] \sin \left[\frac{(4 \ell - 2)\pi}{n+1} \right] - \sin \left[\frac{(4k - 2)\pi}{n+1} \right] \sin \left[\frac{(2 \ell - 1)\pi}{n+1} \right] \right\}^{2}}{(\lambda_{2k - 1} - \zeta)(\lambda_{2\ell - 1} - \zeta)}
\end{split}
\end{equation*}
we obtain \eqref{eq:3.4b}. Considering now $\theta = 0$ and setting
\begin{gather*}
\mathbf{X} = \left[
\begin{array}{cccc}
  2\sqrt{\frac{\vartheta}{n + 1}} \sin \left(\frac{\pi}{n+1} \right) & 2\sqrt{\frac{\vartheta}{n + 1}} \sin \left(\frac{3\pi}{n+1} \right) & \ldots & 2\sqrt{\frac{\vartheta}{n + 1}} \sin \left(\frac{n\pi}{n+1} \right) \\
  2\sqrt{\frac{\vartheta}{n + 1}} \sin \left(\frac{2\pi}{n+1} \right) & 2\sqrt{\frac{\vartheta}{n + 1}} \sin \left(\frac{6\pi}{n+1} \right) & \ldots & 2\sqrt{\frac{\vartheta}{n + 1}} \sin \left[\frac{(4n-2)\pi}{n+1} \right]
\end{array}
\right], \\
\mathbf{Y} = \left[
\begin{array}{cccc}
  2\sqrt{\frac{\vartheta}{n + 1}} \sin \left(\frac{2\pi}{n+1} \right) & 2\sqrt{\frac{\vartheta}{n + 1}} \sin \left(\frac{6\pi}{n+1} \right) & \ldots & 2\sqrt{\frac{\vartheta}{n + 1}} \sin \left[\frac{(4n-2)\pi}{n+1} \right] \\
  2\sqrt{\frac{\vartheta}{n + 1}} \sin \left(\frac{\pi}{n+1} \right) & 2\sqrt{\frac{\vartheta}{n + 1}} \sin \left(\frac{3\pi}{n+1} \right) & \ldots & 2\sqrt{\frac{\vartheta}{n + 1}} \sin \left(\frac{n\pi}{n+1} \right)
\end{array}
\right]
\end{gather*}
we still have that $\zeta$ is an eigenvalue of \eqref{eq:3.4a} if and only if
\begin{equation*}
1 + \sum_{k=1}^{\frac{n}{2}} \sum_{J \in \EuScript{S}\left(k,\frac{n}{2} \right)} \sum_{I \in \EuScript{S}(k,2)} \frac{\det\left(\mathbf{X}[I,J] \right) \det\left(\mathbf{Y}[I,J] \right)}{\prod_{j \in J}(\lambda_{2j-1} - \zeta)} = 0
\end{equation*}
assuming that $\zeta$ is not an eigenvalue of $\mathrm{diag} \left(\lambda_{1},\lambda_{3},\ldots,\lambda_{n-1} \right)$. Hence,
\begin{equation*}
\begin{split}
1 + \sum_{k=1}^{\frac{n}{2}} \sum_{J \in \EuScript{S}\left(k,\frac{n}{2} \right)}  & \sum_{I \in \EuScript{S}(k,2)} \frac{\det\left(\mathbf{X}[I,J] \right) \det\left(\mathbf{Y}[I,J] \right)}{\prod_{j \in J}(\lambda_{2j-1} - \zeta)} = 1 + \frac{8\vartheta}{n+1} \sum_{k=1}^{\frac{n}{2}} \frac{\sin \left[\frac{(2k - 1)\pi}{n+1} \right] \sin \left[\frac{(4k - 2)\pi}{n+1} \right]}{\lambda_{2k - 1} - \zeta} + \\
& - \frac{16\vartheta^{2}}{(n+1)^{2}} \sum_{1 \leqslant k < \ell \leqslant \frac{n}{2}} \frac{\left\{\sin \left[\frac{(2k - 1)\pi}{n+1} \right] \sin \left[\frac{(4 \ell - 2)\pi}{n+1} \right] - \sin \left[\frac{(4k - 2)\pi}{n+1} \right] \sin \left[\frac{(2 \ell - 1)\pi}{n+1} \right] \right\}^{2}}{(\lambda_{2k - 1} - \zeta)(\lambda_{2\ell - 1} - \zeta)}
\end{split}
\end{equation*}
and \eqref{eq:3.4b} is established. Let $\mu_{1} \leqslant \mu_{2} \leqslant \ldots \leqslant \mu_{\frac{n}{2}}$ be the eigenvalues of \eqref{eq:3.4a} and $\lambda_{\tau(1)} \leqslant \lambda_{\tau(3)} \leqslant \ldots \leqslant \lambda_{\tau(n-1)}$ be arranged in non-decreasing order by some bijection $\tau$ defined in $\{1,3, \ldots, n-1\}$. Thus,
\begin{equation}\label{eq:3.8}
\lambda_{\tau(2k-1)} + \lambda_{\min} \left(\theta \mathbf{x} \mathbf{x}^{\top} + \vartheta \mathbf{x} \mathbf{y}^{\top} + \vartheta \mathbf{y} \mathbf{x}^{\top} \right) \leqslant \mu_{k} \leqslant \lambda_{\tau(2k-1)} + \lambda_{\max} \left(\theta \mathbf{x} \mathbf{x}^{\top} + \vartheta \mathbf{x} \mathbf{y}^{\top} + \vartheta \mathbf{y} \mathbf{x}^{\top} \right)
\end{equation}
for each $k=1,\ldots,\tfrac{n}{2}$ (see \cite{Horn13}, page $242$). Since the characteristic polynomial of $\theta \mathbf{x} \mathbf{x}^{\top} + \vartheta \mathbf{x} \mathbf{y}^{\top} + \vartheta \mathbf{y} \mathbf{x}^{\top}$ is
\begin{equation*}
\begin{split}
\det \left[t \mathbf{I}_{\frac{n}{2}} - \theta \mathbf{x} \mathbf{x}^{\top} - \vartheta \mathbf{x} \mathbf{y}^{\top} - \vartheta \mathbf{y} \mathbf{x}^{\top} \right] &= t^{\frac{n}{2} - 2} \Big[t^{2} - \left(\theta \mathbf{x}^{\top} \mathbf{x} + \vartheta \mathbf{y}^{\top} \mathbf{x} + \vartheta \mathbf{x}^{\top} \mathbf{y} \right)t \Big. \\
& \hspace*{3.8cm} \Big. + \vartheta^{2} \left(\mathbf{x}^{\top} \mathbf{y} \right) \left(\mathbf{y}^{\top} \mathbf{x} \right) - \vartheta^{2} \left(\mathbf{x}^{\top} \mathbf{x} \right) \left(\mathbf{y}^{\top} \mathbf{y} \right) \Big] \\
&= t^{\frac{n}{2} - 2} \Big\{t^{2} - \left(\theta \norm{\mathbf{x}}^{2} + 2\vartheta \, \mathbf{x}^{\top} \mathbf{y} \right)t + \vartheta^{2} \left[\left(\mathbf{x}^{\top} \mathbf{y} \right)^{2} - \norm{\mathbf{x}}^{2} \norm{\mathbf{y}}^{2} \right] \Big\}
\end{split}
\end{equation*}
we have that its spectrum is
\begin{equation}\label{eq:3.9}
\mathrm{Spec} \left(\theta \mathbf{x} \mathbf{x}^{\top} + \vartheta \mathbf{x} \mathbf{y}^{\top} + \vartheta \mathbf{y} \mathbf{x}^{\top} \right) = \left\{0, \alpha^{-}, \alpha^{+} \right\}
\end{equation}
where $\alpha^{\pm} := \frac{\theta \norm{\mathbf{x}}^{2} + 2 \vartheta \mathbf{x}^{\top} \mathbf{y} \pm \sqrt{\left(\theta \norm{\mathbf{x}}^{2} + 2 \vartheta \mathbf{x}^{\top} \mathbf{y}\right)^{2} - 4\vartheta^{2} \left[\left(\mathbf{x}^{\top} \mathbf{y} \right)^{2} - \norm{\mathbf{x}}^{2} \norm{\mathbf{y}}^{2} \right]}}{2}$. From the identities,
\begin{gather*}
\sum_{k=1}^{\frac{n}{2}} \sin^{2} \left[\frac{(2k - 1) \pi}{n + 1} \right] = \frac{n + 1}{4} = \sum_{k=1}^{\frac{n}{2}} \sin^{2} \left[\frac{(4k - 2) \pi}{n + 1} \right], \\
\sum_{k=1}^{\frac{n}{2}} \sin \left[\frac{(2k - 1) \pi}{n + 1} \right] \sin \left[\frac{(4k - 2) \pi}{n + 1} \right] = 0
\end{gather*}
it follows $\lVert \mathbf{x} \rVert = \lVert \mathbf{y} \rVert = 1$ and $\mathbf{x}^{\top} \mathbf{y} = 0$. Hence, \eqref{eq:3.8} and \eqref{eq:3.9} yields \eqref{eq:3.4c}. The proofs of the remaining assertions are performed in the same way and so will be omitted.
\end{proof}

The next statement allows us to locate the eigenvalues of $\mathbf{H}_{n}$ providing also explicit bounds for each of them.


\begin{thr}
Let $n \geqslant 5$ be an integer, $a,b,c,d, \xi, \eta \in \mathbb{R}$, $\lambda_{k}$, $k = 1,\ldots,n$ be given by \eqref{eq:2.3} and $\mathbf{H}_{n}$ the $n \times n$ matrix \eqref{eq:1.1}.

\medskip

\noindent \textnormal{(a)} If $n$ is even, the eigenvalues of $\boldsymbol{\Phi}_{\frac{n}{2}}$ in \eqref{eq:2.4d} are not of the form $\lambda_{2k-1}$, $k = 1,\ldots,\frac{n}{2}$ and the eigenvalues of $\boldsymbol{\Psi}_{\frac{n}{2}}$ in \eqref{eq:2.4e} are not of the form $\lambda_{2k}$, $k = 1,\ldots,\frac{n}{2}$ then the eigenvalues of $\mathbf{H}_{n}$ are precisely the zeros of the rational functions $f(t)$ and $g(t)$ given by \eqref{eq:3.4b} and \eqref{eq:3.5b}, respectively. Moreover, if $\mu_{1} \leqslant \mu_{2} \leqslant \ldots \leqslant \mu_{\frac{n}{2}}$ are the zeros of $f(t)$ and $\nu_{1} \leqslant \nu_{2} \leqslant \ldots \leqslant \nu_{\frac{n}{2}}$ are the zeros of $g(t)$ (counting multiplicities in both cases) then $\mu_{k}$, $k=1,\ldots,\tfrac{n}{2}$ and $\nu_{k}$, $k=1,\ldots,\tfrac{n}{2}$ satisfy \eqref{eq:3.4c} and \eqref{eq:3.5c}, respectively.

\medskip

\noindent \textnormal{(b)} If $n$ is odd, the eigenvalues of $\boldsymbol{\Phi}_{\frac{n+1}{2}}$ in \eqref{eq:2.5d} are not of the form $\lambda_{2k-1}$, $k = 1,\ldots,\frac{n+1}{2}$ and the eigenvalues of $\boldsymbol{\Psi}_{\frac{n-1}{2}}$ in \eqref{eq:2.5e} are not of the form $\lambda_{2k}$, $k = 1,\ldots,\frac{n-1}{2}$ then the eigenvalues of $\mathbf{H}_{n}$ are precisely the zeros of the rational functions $f(t)$ and $g(t)$ given by \eqref{eq:3.6b} and \eqref{eq:3.7b}, respectively. Furthermore, if $\mu_{1} \leqslant \mu_{2} \leqslant \ldots \leqslant \mu_{\frac{n+1}{2}}$ are the zeros of $f(t)$ and $\nu_{1} \leqslant \nu_{2} \leqslant \ldots \leqslant \nu_{\frac{n-1}{2}}$ are the zeros of $g(t)$ (counting multiplicities in both cases) then $\mu_{k}$, $k=1,\ldots,\tfrac{n+1}{2}$ and $\nu_{k}$, $k=1,\ldots,\tfrac{n-1}{2}$ satisfy \eqref{eq:3.6c} and \eqref{eq:3.7c}, respectively.
\end{thr}

\begin{proof}
Suppose an integer $n \geqslant 5$, $a,b,c,d, \xi, \eta \in \mathbb{R}$ and $\lambda_{k}$, $k = 1,\ldots,n$ be given by \eqref{eq:2.3}.

\medskip

\noindent (a) According to Lemma~\ref{lem3} and the determinant formula for block-triangular matrices (see \cite{Harville97}, page $185$), the characteristic polynomial of $\mathbf{H}_{n}$, for $n$ even, is
\begin{equation*}
\det \left(t \mathbf{I}_{n} - \mathbf{H}_{n} \right) = \det \left(t \mathbf{I}_{\frac{n}{2}} - \boldsymbol{\Phi}_{\frac{n}{2}} \right) \det \left(t \mathbf{I}_{\frac{n}{2}} - \boldsymbol{\Psi}_{\frac{n}{2}} \right)
\end{equation*}
where $\boldsymbol{\Phi}_{\frac{n}{2}}$ and $\boldsymbol{\Psi}_{\frac{n}{2}}$ are given by \eqref{eq:2.4d} and \eqref{eq:2.4e}, respectively, so that the thesis is a direct consequence of Lemma~\ref{lem4}.

\medskip

\noindent (b) For $n$ odd, we obtain
\begin{equation*}
\det \left(t \mathbf{I}_{n} - \mathbf{H}_{n} \right) = \det \left(t \mathbf{I}_{\frac{n+1}{2}} - \boldsymbol{\Phi}_{\frac{n+1}{2}} \right) \det \left(t \mathbf{I}_{\frac{n-1}{2}} - \boldsymbol{\Psi}_{\frac{n-1}{2}} \right)
\end{equation*}
where $\boldsymbol{\Phi}_{\frac{n+1}{2}}$ and $\boldsymbol{\Psi}_{\frac{n-1}{2}}$ are given by \eqref{eq:2.5d} and \eqref{eq:2.5e}, respectively. The conclusion follows from Lemma~\ref{lem4}.
\end{proof}

\subsection{Eigenvectors of $\mathbf{H}_{n}$} \label{sec:3}

\indent

The decomposition presented in Lemma~\ref{lem3} allows us also to compute eigenvectors for $\mathbf{H}_{n}$ in \eqref{eq:1.1}.


\begin{thr}
Let $n \geqslant 5$ be an integer, $a,b,c,d, \xi, \eta \in \mathbb{R}$, $\lambda_{k}$, $k = 1,\ldots,n$ be given by \eqref{eq:2.3} and $\mathbf{H}_{n}$ the $n \times n$ matrix \eqref{eq:1.1}.

\medskip

\begin{subequations}
\noindent \textnormal{(a)} If $n$ is even, $\mathbf{S}_{n}$ is the $n \times n$ matrix \eqref{eq:2.1}, $\mathbf{P}_{n}$ is the $n \times n$ permutation matrix \eqref{eq:2.4c}, the zeros $\mu_{1}, \ldots, \mu_{\frac{n}{2}}$ of \eqref{eq:3.4b} are not of the form $\lambda_{2k-1}$, $k = 1,\ldots,\frac{n}{2}$, the zeros $\nu_{1}, \ldots, \nu_{\frac{n}{2}}$ of \eqref{eq:3.5b} are not of the form $\lambda_{2k}$, $k = 1,\ldots,\frac{n}{2}$,
\begin{gather*}
\underset{j=1}{\overset{\frac{n}{2}}{\sum}} \left\{\frac{(c + \xi - a) \sin^{2} \left[\frac{(2j - 1)\pi}{n + 1} \right] + (d + \eta - b) \sin \left[\frac{(2j - 1)\pi}{n + 1} \right] \sin \left[\frac{(4j - 2)\pi}{n + 1} \right]}{\mu_{k} - \lambda_{2j-1}} \right\} \neq \frac{n+1}{4}, \\
\underset{j=1}{\overset{\frac{n}{2}}{\sum}} \left[\frac{(c + \xi - a) \sin \left(\frac{2j\pi}{n + 1} \right) \sin \left(\frac{4j\pi}{n + 1} \right) + (d + \eta - b) \sin^{2} \left(\frac{4j\pi}{n + 1} \right)}{\nu_{k} - \lambda_{2j}} \right] \neq \frac{n+1}{4}
\end{gather*}
and $b \neq d + \eta$ then
\begin{equation}\label{eq:3.10a}
    \mathbf{S}_{n} \mathbf{P}_{n} \left[
    \begin{array}{c}
      \frac{2\sin\left(\frac{2\pi}{n+1} \right)}{\sqrt{n+1}(\mu_{k} - \lambda_{1})} + \frac{8 \underset{j=1}{\overset{\frac{n}{2}}{\sum}} \left\{\frac{(c + \xi - a) \sin \left[\frac{(2j - 1)\pi}{n + 1} \right] \sin \left[\frac{(4j - 2)\pi}{n + 1} \right] + (d + \eta - b) \sin^{2} \left[\frac{(4j - 2)\pi}{n + 1} \right]}{\mu_{k} - \lambda_{2j-1}} \right\}}{n + 1 - 4 \underset{j=1}{\overset{\frac{n}{2}}{\sum}} \left\{\frac{(c + \xi - a) \sin^{2} \left[\frac{(2j - 1)\pi}{n + 1} \right] + (d + \eta - b) \sin \left[\frac{(2j - 1)\pi}{n + 1} \right] \sin \left[\frac{(4j - 2)\pi}{n + 1} \right]}{\mu_{k} - \lambda_{2j-1}} \right\}} \frac{\sin \left(\frac{\pi}{n+1} \right)}{\sqrt{n+1}(\mu_{k} - \lambda_{1})} \\[25pt]
      \frac{2\sin\left(\frac{6\pi}{n+1} \right)}{\sqrt{n+1}(\mu_{k} - \lambda_{3})} + \frac{8 \underset{j=1}{\overset{\frac{n}{2}}{\sum}} \left\{\frac{(c + \xi - a) \sin \left[\frac{(2j - 1)\pi}{n + 1} \right] \sin \left[\frac{(4j - 2)\pi}{n + 1} \right] + (d + \eta - b) \sin^{2} \left[\frac{(4j - 2)\pi}{n + 1} \right]}{\mu_{k} - \lambda_{2j-1}} \right\}}{n + 1 - 4 \underset{j=1}{\overset{\frac{n}{2}}{\sum}} \left\{\frac{(c + \xi - a) \sin^{2} \left[\frac{(2j - 1)\pi}{n + 1} \right] + (d + \eta - b) \sin \left[\frac{(2j - 1)\pi}{n + 1} \right] \sin \left[\frac{(4j - 2)\pi}{n + 1} \right]}{\mu_{k} - \lambda_{2j-1}} \right\}} \frac{\sin \left(\frac{3\pi}{n+1} \right)}{\sqrt{n+1}(\mu_{k} - \lambda_{3})} \\
      \vdots \\[3pt]
      \frac{2\sin\left[\frac{(2n-2)\pi}{n+1} \right]}{\sqrt{n+1}(\mu_{k} - \lambda_{n-1})} + \frac{8 \underset{j=1}{\overset{\frac{n}{2}}{\sum}} \left\{\frac{(c + \xi - a) \sin \left[\frac{(2j - 1)\pi}{n + 1} \right] \sin \left[\frac{(4j - 2)\pi}{n + 1} \right] + (d + \eta - b) \sin^{2} \left[\frac{(4j - 2)\pi}{n + 1} \right]}{\mu_{k} - \lambda_{2j-1}} \right\}}{n + 1 - 4 \underset{j=1}{\overset{\frac{n}{2}}{\sum}} \left\{\frac{(c + \xi - a) \sin^{2} \left[\frac{(2j - 1)\pi}{n + 1} \right] + (d + \eta - b) \sin \left[\frac{(2j - 1)\pi}{n + 1} \right] \sin \left[\frac{(4j - 2)\pi}{n + 1} \right]}{\mu_{k} - \lambda_{2j-1}} \right\}} \frac{\sin \left[\frac{(n-1)\pi}{n+1} \right]}{\sqrt{n+1}(\mu_{k} - \lambda_{n-1})} \\[15pt]
      0 \\
      \vdots \\[3pt]
      0
    \end{array}
    \right]
\end{equation}
is an eigenvector of $\mathbf{H}_{n}$ associated to $\mu_{k}$, $k=1,\ldots, \frac{n}{2}$ and \begin{equation}\label{eq:3.10b}
    \mathbf{S}_{n} \mathbf{P}_{n} \left[
    \begin{array}{c}
      0 \\
      \vdots \\[3pt]
      0 \\[5pt]
      \frac{2\sin\left(\frac{4\pi}{n+1} \right)}{\sqrt{n+1}(\nu_{k} - \lambda_{2})} + \frac{8 \underset{j=1}{\overset{\frac{n}{2}}{\sum}} \left[\frac{(c + \xi - a) \sin \left(\frac{2j\pi}{n + 1} \right) \sin \left(\frac{4j\pi}{n + 1} \right) + (d + \eta - b) \sin^{2} \left(\frac{4j\pi}{n + 1} \right)}{\nu_{k} - \lambda_{2j}} \right]}{n + 1 - 4 \underset{j=1}{\overset{\frac{n}{2}}{\sum}} \left[\frac{(c + \xi - a) \sin \left(\frac{2j\pi}{n + 1} \right) \sin \left(\frac{4j\pi}{n + 1} \right) + (d + \eta - b) \sin^{2} \left(\frac{4j\pi}{n + 1} \right)}{\nu_{k} - \lambda_{2j}} \right]} \frac{\sin \left(\frac{2\pi}{n+1} \right)}{\sqrt{n+1}(\nu_{k} - \lambda_{2})} \\[20pt]
      \frac{2\sin\left(\frac{8\pi}{n+1} \right)}{\sqrt{n+1}(\nu_{k} - \lambda_{4})} + \frac{8 \underset{j=1}{\overset{\frac{n}{2}}{\sum}} \left[\frac{(c + \xi - a) \sin \left(\frac{2j\pi}{n + 1} \right) \sin \left(\frac{4j\pi}{n + 1} \right) + (d + \eta - b) \sin^{2} \left(\frac{4j\pi}{n + 1} \right)}{\nu_{k} - \lambda_{2j}} \right]}{n + 1 - 4 \underset{j=1}{\overset{\frac{n}{2}}{\sum}} \left[\frac{(c + \xi - a) \sin \left(\frac{2j\pi}{n + 1} \right) \sin \left(\frac{4j\pi}{n + 1} \right) + (d + \eta - b) \sin^{2} \left(\frac{4j\pi}{n + 1} \right)}{\nu_{k} - \lambda_{2j}} \right]} \frac{\sin \left(\frac{4\pi}{n+1} \right)}{\sqrt{n+1}(\nu_{k} - \lambda_{4})} \\
      \vdots \\[5pt]
      \frac{2\sin\left(\frac{2n\pi}{n+1} \right)}{\sqrt{n+1}(\nu_{k} - \lambda_{n})} + \frac{8 \underset{j=1}{\overset{\frac{n}{2}}{\sum}} \left[\frac{(c + \xi - a) \sin \left(\frac{2j\pi}{n + 1} \right) \sin \left(\frac{4j\pi}{n + 1} \right) + (d + \eta - b) \sin^{2} \left(\frac{4j\pi}{n + 1} \right)}{\nu_{k} - \lambda_{2j}} \right]}{n + 1 - 4 \underset{j=1}{\overset{\frac{n}{2}}{\sum}} \left[\frac{(c + \xi - a) \sin \left(\frac{2j\pi}{n + 1} \right) \sin \left(\frac{4j\pi}{n + 1} \right) + (d + \eta - b) \sin^{2} \left(\frac{4j\pi}{n + 1} \right)}{\nu_{k} - \lambda_{2j}} \right]} \frac{\sin \left(\frac{n\pi}{n+1} \right)}{\sqrt{n+1}(\nu_{k} - \lambda_{n})}
    \end{array}
    \right]
\end{equation}
is an eigenvector of $\mathbf{H}_{n}$ associated to $\nu_{k}$, $k=1,\ldots, \frac{n}{2}$.
\end{subequations}

\medskip

\begin{subequations}
\noindent \textnormal{(b)} If $n$ is odd, $\mathbf{S}_{n}$ is the $n \times n$ matrix \eqref{eq:2.1}, $\mathbf{P}_{n}$ is the $n \times n$ permutation matrix \eqref{eq:2.4c}, the zeros $\mu_{1}, \ldots, \mu_{\frac{n+1}{2}}$ of \eqref{eq:3.6b} are not of the form $\lambda_{2k-1}$, $k = 1,\ldots,\frac{n+1}{2}$, the zeros $\nu_{1}, \ldots, \nu_{\frac{n-1}{2}}$ of \eqref{eq:3.7b} are not of the form $\lambda_{2k}$, $k = 1,\ldots,\frac{n-1}{2}$,
\begin{gather*}
\underset{j=1}{\overset{\frac{n+1}{2}}{\sum}} \left\{\frac{(c + \xi - a) \sin^{2} \left[\frac{(2j - 1)\pi}{n + 1} \right] + (d + \eta - b) \sin \left[\frac{(2j - 1)\pi}{n + 1} \right] \sin \left[\frac{(4j - 2)\pi}{n + 1} \right]}{\mu_{k} - \lambda_{2j-1}} \right\} \neq \frac{n+1}{4}, \\
\underset{j=1}{\overset{\frac{n-1}{2}}{\sum}} \left[\frac{(c + \xi - a) \sin \left(\frac{2j\pi}{n + 1} \right) \sin \left(\frac{4j\pi}{n + 1} \right) + (d + \eta - b) \sin^{2} \left(\frac{4j\pi}{n + 1} \right)}{\nu_{k} - \lambda_{2j}} \right] \neq \frac{n+1}{4}
\end{gather*}
and $b \neq d + \eta$ then
\begin{equation}\label{eq:3.11a}
    \mathbf{S}_{n} \mathbf{P}_{n} \left[
    \begin{array}{c}
      \frac{2\sin\left(\frac{2\pi}{n+1} \right)}{\sqrt{n+1}(\mu_{k} - \lambda_{1})} + \frac{8 \underset{j=1}{\overset{\frac{n+1}{2}}{\sum}} \left\{\frac{(c + \xi - a) \sin \left[\frac{(2j - 1)\pi}{n + 1} \right] \sin \left[\frac{(4j - 2)\pi}{n + 1} \right] + (d + \eta - b) \sin^{2} \left[\frac{(4j - 2)\pi}{n + 1} \right]}{\mu_{k} - \lambda_{2j-1}} \right\}}{n + 1 - 4 \underset{j=1}{\overset{\frac{n+1}{2}}{\sum}} \left\{\frac{(c + \xi - a) \sin^{2} \left[\frac{(2j - 1)\pi}{n + 1} \right] + (d + \eta - b) \sin \left[\frac{(2j - 1)\pi}{n + 1} \right] \sin \left[\frac{(4j - 2)\pi}{n + 1} \right]}{\mu_{k} - \lambda_{2j-1}} \right\}} \frac{\sin \left(\frac{\pi}{n+1} \right)}{\sqrt{n+1}(\mu_{k} - \lambda_{1})} \\[25pt]
      \frac{2\sin\left(\frac{6\pi}{n+1} \right)}{\sqrt{n+1}(\mu_{k} - \lambda_{3})} + \frac{8 \underset{j=1}{\overset{\frac{n+1}{2}}{\sum}} \left\{\frac{(c + \xi - a) \sin \left[\frac{(2j - 1)\pi}{n + 1} \right] \sin \left[\frac{(4j - 2)\pi}{n + 1} \right] + (d + \eta - b) \sin^{2} \left[\frac{(4j - 2)\pi}{n + 1} \right]}{\mu_{k} - \lambda_{2j-1}} \right\}}{n + 1 - 4 \underset{j=1}{\overset{\frac{n+1}{2}}{\sum}} \left\{\frac{(c + \xi - a) \sin^{2} \left[\frac{(2j - 1)\pi}{n + 1} \right] + (d + \eta - b) \sin \left[\frac{(2j - 1)\pi}{n + 1} \right] \sin \left[\frac{(4j - 2)\pi}{n + 1} \right]}{\mu_{k} - \lambda_{2j-1}} \right\}} \frac{\sin \left(\frac{3\pi}{n+1} \right)}{\sqrt{n+1}(\mu_{k} - \lambda_{3})} \\
      \vdots \\[3pt]
      \frac{2\sin\left(\frac{2n\pi}{n+1} \right)}{\sqrt{n+1}(\mu_{k} - \lambda_{n-1})} + \frac{8 \underset{j=1}{\overset{\frac{n+1}{2}}{\sum}} \left\{\frac{(c + \xi - a) \sin \left[\frac{(2j - 1)\pi}{n + 1} \right] \sin \left[\frac{(4j - 2)\pi}{n + 1} \right] + (d + \eta - b) \sin^{2} \left[\frac{(4j - 2)\pi}{n + 1} \right]}{\mu_{k} - \lambda_{2j-1}} \right\}}{n + 1 - 4 \underset{j=1}{\overset{\frac{n+1}{2}}{\sum}} \left\{\frac{(c + \xi - a) \sin^{2} \left[\frac{(2j - 1)\pi}{n + 1} \right] + (d + \eta - b) \sin \left[\frac{(2j - 1)\pi}{n + 1} \right] \sin \left[\frac{(4j - 2)\pi}{n + 1} \right]}{\mu_{k} - \lambda_{2j-1}} \right\}} \frac{\sin \left(\frac{n\pi}{n+1} \right)}{\sqrt{n+1}(\mu_{k} - \lambda_{n-1})} \\[15pt]
      0 \\
      \vdots \\[3pt]
      0
    \end{array}
    \right]
\end{equation}
is an eigenvector of $\mathbf{H}_{n}$ associated to $\mu_{k}$, $k=1,\ldots, \frac{n+1}{2}$ and \begin{equation}\label{eq:3.11b}
    \mathbf{S}_{n} \mathbf{P}_{n} \left[
    \begin{array}{c}
      0 \\
      \vdots \\[3pt]
      0 \\[5pt]
      \frac{2\sin\left(\frac{4\pi}{n+1} \right)}{\sqrt{n+1}(\nu_{k} - \lambda_{2})} + \frac{8 \underset{j=1}{\overset{\frac{n-1}{2}}{\sum}} \left[\frac{(c + \xi - a) \sin \left(\frac{2j\pi}{n + 1} \right) \sin \left(\frac{4j\pi}{n + 1} \right) + (d + \eta - b) \sin^{2} \left(\frac{4j\pi}{n + 1} \right)}{\nu_{k} - \lambda_{2j}} \right]}{n + 1 - 4 \underset{j=1}{\overset{\frac{n-1}{2}}{\sum}} \left[\frac{(c + \xi - a) \sin \left(\frac{2j\pi}{n + 1} \right) \sin \left(\frac{4j\pi}{n + 1} \right) + (d + \eta - b) \sin^{2} \left(\frac{4j\pi}{n + 1} \right)}{\nu_{k} - \lambda_{2j}} \right]} \frac{\sin \left(\frac{2\pi}{n+1} \right)}{\sqrt{n+1}(\nu_{k} - \lambda_{2})} \\[20pt]
      \frac{2\sin\left(\frac{8\pi}{n+1} \right)}{\sqrt{n+1}(\nu_{k} - \lambda_{4})} + \frac{8 \underset{j=1}{\overset{\frac{n-1}{2}}{\sum}} \left[\frac{(c + \xi - a) \sin \left(\frac{2j\pi}{n + 1} \right) \sin \left(\frac{4j\pi}{n + 1} \right) + (d + \eta - b) \sin^{2} \left(\frac{4j\pi}{n + 1} \right)}{\nu_{k} - \lambda_{2j}} \right]}{n + 1 - 4 \underset{j=1}{\overset{\frac{n-1}{2}}{\sum}} \left[\frac{(c + \xi - a) \sin \left(\frac{2j\pi}{n + 1} \right) \sin \left(\frac{4j\pi}{n + 1} \right) + (d + \eta - b) \sin^{2} \left(\frac{4j\pi}{n + 1} \right)}{\nu_{k} - \lambda_{2j}} \right]} \frac{\sin \left(\frac{4\pi}{n+1} \right)}{\sqrt{n+1}(\nu_{k} - \lambda_{4})} \\
      \vdots \\[5pt]
      \frac{2\sin\left[\frac{2(n-1)\pi}{n+1} \right]}{\sqrt{n+1}(\nu_{k} - \lambda_{n})} + \frac{8 \underset{j=1}{\overset{\frac{n-1}{2}}{\sum}} \left[\frac{(c + \xi - a) \sin \left(\frac{2j\pi}{n + 1} \right) \sin \left(\frac{4j\pi}{n + 1} \right) + (d + \eta - b) \sin^{2} \left(\frac{4j\pi}{n + 1} \right)}{\nu_{k} - \lambda_{2j}} \right]}{n + 1 - 4 \underset{j=1}{\overset{\frac{n-1}{2}}{\sum}} \left[\frac{(c + \xi - a) \sin \left(\frac{2j\pi}{n + 1} \right) \sin \left(\frac{4j\pi}{n + 1} \right) + (d + \eta - b) \sin^{2} \left(\frac{4j\pi}{n + 1} \right)}{\nu_{k} - \lambda_{2j}} \right]} \frac{\sin \left[\frac{(n-1)\pi}{n+1} \right]}{\sqrt{n+1}(\nu_{k} - \lambda_{n})}
    \end{array}
    \right]
\end{equation}
is an eigenvector of $\mathbf{H}_{n}$ associated to $\nu_{k}$, $k=1,\ldots, \frac{n-1}{2}$.
\end{subequations}
\end{thr}

\begin{proof}
Since both assertions can be proven in the same way, we only prove (a). Let $n \geqslant 5$ be even. We can rewrite the matricial equation $(\mu_{k} \mathbf{I}_{n} - \mathbf{H}_{n}) \mathbf{q} = \mathbf{0}$ as
\begin{equation}\label{eq:3.12}
\mathbf{S}_{n} \mathbf{P}_{n} \left[
\setlength{\extrarowheight}{2pt}
\begin{array}{c|c}
\mu_{k} \mathbf{I}_{\frac{n}{2}} - \boldsymbol{\Phi}_{\frac{n}{2}} & \mathbf{O} \\[2pt] \hline
\mathbf{O} & \mu_{k} \mathbf{I}_{\frac{n}{2}} - \boldsymbol{\Psi}_{\frac{n}{2}}
\end{array}
\right] \mathbf{P}_{n}^{\top} \mathbf{S}_{n} \mathbf{q} = \mathbf{0}
\end{equation}
where $\mathbf{S}_{n}$ is the matrix \eqref{eq:2.1}, $\mathbf{P}_{n}$ is the permutation matrix \eqref{eq:2.4c}, $\boldsymbol{\Phi}_{\frac{n}{2}}$ and $\boldsymbol{\Psi}_{\frac{n}{2}}$ are given by \eqref{eq:2.4d} and \eqref{eq:2.4e}, respectively. Thus,
\begin{gather*}
\left[\mu_{k} \mathbf{I}_{\frac{n}{2}} - \mathrm{diag} \left(\lambda_{1},\lambda_{3},\ldots,\lambda_{n-1} \right) - (c + \xi - a) \mathbf{x} \mathbf{x}^{\top} - (d + \eta - b) \mathbf{x} \mathbf{y}^{\top} - (d + \eta - b) \mathbf{y} \mathbf{x}^{\top} \right] \mathbf{q}_{1} = \mathbf{0}, \\
\left[\mu_{k} \mathbf{I}_{\frac{n}{2}} - \mathrm{diag} \left(\lambda_{2},\lambda_{4},\ldots,\lambda_{n} \right) - (c + \xi - a) \mathbf{v} \mathbf{v}^{\top} - (d + \eta - b) \mathbf{v} \mathbf{w}^{\top} - (d + \eta - b) \mathbf{w} \mathbf{v}^{\top} \right] \mathbf{q}_{2} = \mathbf{0}, \\
\left[
\begin{array}{c}
\mathbf{q}_{1} \\
\mathbf{q}_{2}
\end{array}
\right] = \mathbf{P}_{n}^{\top} \mathbf{S}_{n} \mathbf{q}
\end{gather*}
that is,
\begin{gather*}
\mathbf{q}_{1} = \alpha \left[\mu_{k} \mathbf{I}_{\frac{n}{2}} - \mathrm{diag} \left(\lambda_{1},\lambda_{3},\ldots,\lambda_{n-1} \right) - (c + \xi - a) \mathbf{x} \mathbf{x}^{\top} - (d + \eta - b) \mathbf{x} \mathbf{y}^{\top} \right]^{-1}  \mathbf{y}, \\
\mathbf{q}_{2} = \mathbf{0}
\end{gather*}
for $\alpha \neq 0$ (see \cite{Bunch78}, page $41$) and
\begin{equation*}
\mathbf{q} =  \mathbf{S}_{n} \mathbf{P}_{n}  \left[
\begin{array}{c}
\alpha \left[\mu_{k} \mathbf{I}_{\frac{n}{2}} - \mathrm{diag} \left(\lambda_{1},\lambda_{3},\ldots,\lambda_{n-1} \right) - (c + \xi - a) \mathbf{x} \mathbf{x}^{\top} - (d + \eta - b) \mathbf{x} \mathbf{y}^{\top} \right]^{-1}  \mathbf{y} \\
\mathbf{0}
\end{array}
\right]
\end{equation*}
is a nontrivial solution of \eqref{eq:3.12}. Thus, choosing $\alpha = 1$ we conclude that \eqref{eq:3.10a} is an eigenvector of $\mathbf{H}_{n}$ associated to the eigenvalue $\mu_{k}$. Similarly, from $(\nu_{k} \mathbf{I}_{n} - \mathbf{H}_{n}) \mathbf{q} = \mathbf{0}$ we have
\begin{equation*}
\mathbf{S}_{n} \mathbf{P}_{n} \left[
\setlength{\extrarowheight}{2pt}
\begin{array}{c|c}
\nu_{k} \mathbf{I}_{\frac{n}{2}} - \boldsymbol{\Phi}_{\frac{n}{2}} & \mathbf{O} \\[2pt] \hline
\mathbf{O} & \nu_{k} \mathbf{I}_{\frac{n}{2}} - \boldsymbol{\Psi}_{\frac{n}{2}}
\end{array}
\right] \mathbf{P}_{n}^{\top} \mathbf{S}_{n} \mathbf{q} = \mathbf{0}
\end{equation*}
and
\begin{equation*}
\mathbf{q} =  \mathbf{S}_{n} \mathbf{P}_{n} \left[
\begin{array}{c}
\mathbf{0} \\
\alpha \left[\nu_{k} \mathbf{I}_{\frac{n}{2}} - \mathrm{diag} \left(\lambda_{2},\lambda_{4},\ldots,\lambda_{n} \right) - (c + \xi - a) \mathbf{v} \mathbf{v}^{\top} - (d + \eta - b) \mathbf{v} \mathbf{w}^{\top} \right]^{-1} \mathbf{w}
\end{array}
\right],
\end{equation*}
for $\alpha \neq 0$, which is an eigenvector of $\mathbf{H}_{n}$ associated to the eigenvalue $\nu_{k}$.
\end{proof}

\subsection{Expression of $\mathbf{H}_{n}^{-1}$} \label{sec:4}

\indent

The orthogonal block diagonalization presented in Lemma~\ref{lem3} and Miller's formula for the inverse of the sum of nonsingular matrices lead us to an explicit expression for the inverse of $\mathbf{H}_{n}$.


\begin{thr}
Let $n \geqslant 5$ be an integer, $a,b,c,d, \xi, \eta \in \mathbb{C}$, $\lambda_{k}$, $k = 1,\ldots,n$ be given by \eqref{eq:2.3} and $\mathbf{H}_{n}$ the $n \times n$ matrix \eqref{eq:1.1}. If $\lambda_{k} \neq 0$ for every $k=1,\ldots,n$, $\mathbf{H}_{n}$ is nonsingular and:

\medskip

\noindent \textnormal{(a)} $n$ is even then
\begin{equation*}
\mathbf{H}_{n}^{-1} = \mathbf{S}_{n} \mathbf{P}_{n} \left[
\begin{array}{cc} \mathbf{Q}_{\frac{n}{2}} & \mathbf{O} \\
\mathbf{O} & \mathbf{R}_{\frac{n}{2}}
\end{array}
\right] \mathbf{P}_{n}^{\top} \mathbf{S}_{n}
\end{equation*}
where $\mathbf{S}_{n}$ is the $n \times n$ matrix \eqref{eq:2.1}, $\mathbf{P}_{n}$ is the $n \times n$ permutation matrix \eqref{eq:2.4c},
\begin{subequations}
\begin{equation}\label{eq:3.13a}
\begin{split}
\mathbf{Q}_{\frac{n}{2}} = \boldsymbol{\Upsilon}_{\frac{n}{2}}^{-1} -  & \tfrac{(d + \eta - b) + (d + \eta - b)^{2} \mathbf{y}^{\top} \boldsymbol{\Upsilon}_{\frac{n}{2}}^{-1} \mathbf{x}}{\rho} \boldsymbol{\Upsilon}_{\frac{n}{2}}^{-1} \left(\mathbf{y} \mathbf{x}^{\top} + \mathbf{x} \mathbf{y}^{\top} \right) \boldsymbol{\Upsilon}_{\frac{n}{2}}^{-1} + \\
& \tfrac{(d + \eta - b)^{2}\mathbf{y}^{\top} \boldsymbol{\Upsilon}_{\frac{n}{2}}^{-1} \mathbf{y} - (c + \xi - a)}{\rho} \boldsymbol{\Upsilon}_{\frac{n}{2}}^{-1} \mathbf{x} \mathbf{x}^{\top} \boldsymbol{\Upsilon}_{\frac{n}{2}}^{-1} + \tfrac{(d + \eta - b)^{2} \mathbf{x}^{\top} \boldsymbol{\Upsilon}_{\frac{n}{2}}^{-1} \mathbf{x}}{\rho} \boldsymbol{\Upsilon}_{\frac{n}{2}}^{-1} \mathbf{y} \mathbf{y}^{\top} \boldsymbol{\Upsilon}_{\frac{n}{2}}^{-1},
\end{split}
\end{equation}
with $\boldsymbol{\Upsilon}_{\frac{n}{2}} := \mathrm{diag} \left(\lambda_{1},\lambda_{3},\ldots,\lambda_{n-1} \right)$, $\mathbf{x}, \mathbf{y}$ given by \eqref{eq:2.4a},
\begin{equation}\label{eq:3.13b}
\begin{split}
\rho = 1 + (c + \xi - a) \mathbf{x}^{\top} & \boldsymbol{\Upsilon}_{\frac{n}{2}}^{-1} \mathbf{x} + 2 (d + \eta - b) \mathbf{y}^{\top} \boldsymbol{\Upsilon}_{\frac{n}{2}}^{-1} \mathbf{x} + \\
& (d + \eta - b)^{2} \left[\big(\mathbf{x}^{\top} \boldsymbol{\Upsilon}_{\frac{n}{2}}^{-1} \mathbf{y} \big)^{2} - \big(\mathbf{x}^{\top} \boldsymbol{\Upsilon}_{\frac{n}{2}}^{-1} \mathbf{x} \big) \big(\mathbf{y}^{\top} \boldsymbol{\Upsilon}_{\frac{n}{2}}^{-1} \mathbf{y} \big) \right],
\end{split}
\end{equation}
and
\begin{equation}\label{eq:3.13c}
\begin{split}
\mathbf{R}_{\frac{n}{2}} = \boldsymbol{\Delta}_{\frac{n}{2}}^{-1} -  & \tfrac{(d + \eta - b) + (d + \eta - b)^{2} \mathbf{w}^{\top} \boldsymbol{\Delta}_{\frac{n}{2}}^{-1} \mathbf{v}}{\rho} \boldsymbol{\Delta}_{\frac{n}{2}}^{-1} \left(\mathbf{w} \mathbf{v}^{\top} + \mathbf{v} \mathbf{w}^{\top} \right) \boldsymbol{\Delta}_{\frac{n}{2}}^{-1} + \\
& \tfrac{(d + \eta - b)^{2}\mathbf{w}^{\top} \boldsymbol{\Delta}_{\frac{n}{2}}^{-1} \mathbf{w} - (c + \xi - a)}{\rho} \boldsymbol{\Delta}_{\frac{n}{2}}^{-1} \mathbf{v} \mathbf{v}^{\top} \boldsymbol{\Delta}_{\frac{n}{2}}^{-1} + \tfrac{(d + \eta - b)^{2} \mathbf{v}^{\top} \boldsymbol{\Delta}_{\frac{n}{2}}^{-1} \mathbf{v}}{\rho} \boldsymbol{\Delta}_{\frac{n}{2}}^{-1} \mathbf{w} \mathbf{w}^{\top} \boldsymbol{\Delta}_{\frac{n}{2}}^{-1},
\end{split}
\end{equation}
with $\boldsymbol{\Delta}_{\frac{n}{2}} := \mathrm{diag} \left(\lambda_{2},\lambda_{4},\ldots,\lambda_{n} \right)$, $\mathbf{v}, \mathbf{w}$ given by \eqref{eq:2.5a} and
\begin{equation}\label{eq:3.13d}
\begin{split}
\varrho = 1 + (c + \xi - a) \mathbf{v}^{\top} & \boldsymbol{\Delta}_{\frac{n}{2}}^{-1} \mathbf{v} + 2 (d + \eta - b) \mathbf{w}^{\top} \boldsymbol{\Delta}_{\frac{n}{2}}^{-1} \mathbf{v} + \\
& (d + \eta - b)^{2} \left[\big(\mathbf{v}^{\top} \boldsymbol{\Delta}_{\frac{n}{2}}^{-1} \mathbf{w} \big)^{2} - \big(\mathbf{v}^{\top} \boldsymbol{\Delta}_{\frac{n}{2}}^{-1} \mathbf{v} \big) \big(\mathbf{w}^{\top} \boldsymbol{\Delta}_{\frac{n}{2}}^{-1} \mathbf{w} \big) \right].
\end{split}
\end{equation}
\end{subequations}

\medskip

\noindent \textnormal{(b)} $n$ is odd then
\begin{equation*}
\mathbf{H}_{n}^{-1} = \mathbf{S}_{n} \mathbf{P}_{n} \left[
\begin{array}{cc} \mathbf{Q}_{\frac{n+1}{2}} & \mathbf{O} \\
\mathbf{O} & \mathbf{R}_{\frac{n-1}{2}}
\end{array}
\right] \mathbf{P}_{n}^{\top} \mathbf{S}_{n}
\end{equation*}
where $\mathbf{S}_{n}$ is the $n \times n$ matrix \eqref{eq:2.1}, $\mathbf{P}_{n}$ is the $n \times n$ permutation matrix \eqref{eq:2.5c},
\begin{subequations}
\begin{equation}\label{eq:3.14a}
\begin{split}
\mathbf{Q}_{\frac{n+1}{2}} =  & \boldsymbol{\Upsilon}_{\frac{n+1}{2}}^{-1} - \tfrac{(d + \eta - b) + (d + \eta - b)^{2} \mathbf{y}^{\top} \boldsymbol{\Upsilon}_{\frac{n+1}{2}}^{-1} \mathbf{x}}{\rho} \boldsymbol{\Upsilon}_{\frac{n+1}{2}}^{-1} \left(\mathbf{y} \mathbf{x}^{\top} + \mathbf{x} \mathbf{y}^{\top} \right) \boldsymbol{\Upsilon}_{\frac{n+1}{2}}^{-1} + \\
& \tfrac{(d + \eta - b)^{2}\mathbf{y}^{\top} \boldsymbol{\Upsilon}_{\frac{n+1}{2}}^{-1} \mathbf{y} - (c + \xi - a)}{\rho} \boldsymbol{\Upsilon}_{\frac{n+1}{2}}^{-1} \mathbf{x} \mathbf{x}^{\top} \boldsymbol{\Upsilon}_{\frac{n+1}{2}}^{-1} + \tfrac{(d + \eta - b)^{2} \mathbf{x}^{\top} \boldsymbol{\Upsilon}_{\frac{n+1}{2}}^{-1} \mathbf{x}}{\rho} \boldsymbol{\Upsilon}_{\frac{n+1}{2}}^{-1} \mathbf{y} \mathbf{y}^{\top} \boldsymbol{\Upsilon}_{\frac{n+1}{2}}^{-1},
\end{split}
\end{equation}
with $\boldsymbol{\Upsilon}_{\frac{n+1}{2}} := \mathrm{diag} \left(\lambda_{1},\lambda_{3},\ldots,\lambda_{n} \right)$, $\mathbf{x}, \mathbf{y}$ given by \eqref{eq:2.5a},
\begin{equation}\label{eq:3.14b}
\begin{split}
\rho = 1 + (c + \xi - a) \mathbf{x}^{\top} & \boldsymbol{\Upsilon}_{\frac{n+1}{2}}^{-1} \mathbf{x} + 2 (d + \eta - b) \mathbf{y}^{\top} \boldsymbol{\Upsilon}_{\frac{n+1}{2}}^{-1} \mathbf{x} + \\
& (d + \eta - b)^{2} \left[\big(\mathbf{x}^{\top} \boldsymbol{\Upsilon}_{\frac{n+1}{2}}^{-1} \mathbf{y} \big)^{2} - \big(\mathbf{x}^{\top} \boldsymbol{\Upsilon}_{\frac{n+1}{2}}^{-1} \mathbf{x} \big) \big(\mathbf{y}^{\top} \boldsymbol{\Upsilon}_{\frac{n+1}{2}}^{-1} \mathbf{y} \big) \right],
\end{split}
\end{equation}
and
\begin{equation}\label{eq:3.14c}
\begin{split}
\mathbf{R}_{\frac{n-1}{2}} =  & \boldsymbol{\Delta}_{\frac{n-1}{2}}^{-1} - \tfrac{(d + \eta - b) + (d + \eta - b)^{2} \mathbf{w}^{\top} \boldsymbol{\Delta}_{\frac{n-1}{2}}^{-1} \mathbf{v}}{\rho} \boldsymbol{\Delta}_{\frac{n-1}{2}}^{-1} \left(\mathbf{w} \mathbf{v}^{\top} + \mathbf{v} \mathbf{w}^{\top} \right) \boldsymbol{\Delta}_{\frac{n-1}{2}}^{-1} + \\
& \tfrac{(d + \eta - b)^{2}\mathbf{w}^{\top} \boldsymbol{\Delta}_{\frac{n-1}{2}}^{-1} \mathbf{w} - (c + \xi - a)}{\rho} \boldsymbol{\Delta}_{\frac{n-1}{2}}^{-1} \mathbf{v} \mathbf{v}^{\top} \boldsymbol{\Delta}_{\frac{n-1}{2}}^{-1} + \tfrac{(d + \eta - b)^{2} \mathbf{v}^{\top} \boldsymbol{\Delta}_{\frac{n-1}{2}}^{-1} \mathbf{v}}{\rho} \boldsymbol{\Delta}_{\frac{n-1}{2}}^{-1} \mathbf{w} \mathbf{w}^{\top} \boldsymbol{\Delta}_{\frac{n-1}{2}}^{-1},
\end{split}
\end{equation}
with $\boldsymbol{\Delta}_{\frac{n-1}{2}} := \mathrm{diag} \left(\lambda_{2},\lambda_{4},\ldots,\lambda_{n-1} \right)$, $\mathbf{v}, \mathbf{w}$ in \eqref{eq:2.5b},
\begin{equation}\label{eq:3.14d}
\begin{split}
\varrho = 1 + (c + \xi - a) \mathbf{v}^{\top} & \boldsymbol{\Delta}_{\frac{n-1}{2}}^{-1} \mathbf{v} + 2 (d + \eta - b) \mathbf{w}^{\top} \boldsymbol{\Delta}_{\frac{n-1}{2}}^{-1} \mathbf{v} + \\
& (d + \eta - b)^{2} \left[\big(\mathbf{v}^{\top} \boldsymbol{\Delta}_{\frac{n-1}{2}}^{-1} \mathbf{w} \big)^{2} - \big(\mathbf{v}^{\top} \boldsymbol{\Delta}_{\frac{n-1}{2}}^{-1} \mathbf{v} \big) \big(\mathbf{w}^{\top} \boldsymbol{\Delta}_{\frac{n-1}{2}}^{-1} \mathbf{w} \big) \right].
\end{split}
\end{equation}
\end{subequations}
\end{thr}

\begin{proof}
Consider an even integer $n \geqslant 5$, $a,b,c,d, \xi, \eta \in \mathbb{C}$, $\lambda_{k} \neq 0$, $k = 1,\ldots,n$ be given by \eqref{eq:2.3} and $\mathbf{H}_{n}$ in \eqref{eq:1.1} nonsingular. Recall that if $\mathbf{H}_{n}$ is nonsingular then $\rho$ and $\varrho$ in \eqref{eq:3.13b} and \eqref{eq:3.13d}, respectively, are both nonzero. Setting $\theta := c + \xi - a$, $\vartheta := d + \eta - b$ and assuming that conditions \eqref{eq:3.1a} and \eqref{eq:3.1b} are satisfied (note that \eqref{eq:3.1c} corresponds to $\rho \neq 0$) we have, from the main result of \cite{Miller81} (see \cite{Miller81}, pages $69$ and $70$),
\begin{equation*}
\big(\boldsymbol{\Upsilon}_{\frac{n}{2}} + \theta \mathbf{x} \mathbf{x}^{\top} \big)^{-1} = \boldsymbol{\Upsilon}_{\frac{n}{2}}^{-1} - \tfrac{\theta}{1 + \theta \mathbf{x}^{\top} \boldsymbol{\Upsilon}_{\frac{n}{2}}^{-1} \mathbf{x}} \boldsymbol{\Upsilon}_{\frac{n}{2}}^{-1} \mathbf{x} \mathbf{x}^{\top} \boldsymbol{\Upsilon}_{\frac{n}{2}}^{-1},
\end{equation*}
\begin{equation*}
\begin{split}
&\big(\boldsymbol{\Upsilon}_{\frac{n}{2}} + \theta \mathbf{x} \mathbf{x}^{\top} + \vartheta \mathbf{x} \mathbf{y}^{\top} \big)^{-1} = \\
& \hspace*{2.0cm} = \big(\boldsymbol{\Upsilon}_{\frac{n}{2}} + \theta \mathbf{x} \mathbf{x}^{\top} \big)^{-1} - \tfrac{\vartheta}{1 + \vartheta \mathbf{y}^{\top} \big(\boldsymbol{\Upsilon}_{\frac{n}{2}} + \theta \mathbf{x} \mathbf{x}^{\top} \big)^{-1} \mathbf{x}} \big(\boldsymbol{\Upsilon}_{\frac{n}{2}} + \theta \mathbf{x} \mathbf{x}^{\top} \big)^{-1} \mathbf{x} \mathbf{y}^{\top} \big(\boldsymbol{\Upsilon}_{\frac{n}{2}} + \theta \mathbf{x} \mathbf{x}^{\top} \big)^{-1} \\
& \hspace*{2.0cm} = \boldsymbol{\Upsilon}_{\frac{n}{2}}^{-1} - \tfrac{\theta}{1 + \theta \mathbf{x}^{\top} \boldsymbol{\Upsilon}_{\frac{n}{2}}^{-1} \mathbf{x} + \vartheta \mathbf{y}^{\top} \boldsymbol{\Upsilon}_{\frac{n}{2}}^{-1} \mathbf{x}} \boldsymbol{\Upsilon}_{\frac{n}{2}}^{-1} \mathbf{x} \mathbf{x}^{\top} \boldsymbol{\Upsilon}_{\frac{n}{2}}^{-1} - \tfrac{\vartheta}{1 + \theta \mathbf{x}^{\top} \boldsymbol{\Upsilon}_{\frac{n}{2}}^{-1} \mathbf{x} + \vartheta \mathbf{y}^{\top} \boldsymbol{\Upsilon}_{\frac{n}{2}}^{-1} \mathbf{x}} \boldsymbol{\Upsilon}_{\frac{n}{2}}^{-1} \mathbf{x} \mathbf{y}^{\top} \boldsymbol{\Upsilon}_{\frac{n}{2}}^{-1}
\end{split}
\end{equation*}
and
\begin{equation}\label{eq:3.15}
\begin{split}
\big(\boldsymbol{\Upsilon}_{\frac{n}{2}} + \theta \mathbf{x} \mathbf{x}^{\top} + \vartheta \mathbf{x} \mathbf{y}^{\top} + \vartheta \mathbf{y} \mathbf{x}^{\top} \big)^{-1} &= \\
& \hspace*{-5.0cm} = \big(\boldsymbol{\Upsilon}_{\frac{n}{2}} + \theta \mathbf{x} \mathbf{x}^{\top} + \vartheta \mathbf{x} \mathbf{y}^{\top} \big)^{-1} + \\
&\hspace*{-2.5cm} - \tfrac{\vartheta}{1 + \vartheta \mathbf{x}^{\top} \big(\boldsymbol{\Upsilon}_{\frac{n}{2}} + \theta \mathbf{x} \mathbf{x}^{\top} + \vartheta \mathbf{x} \mathbf{y}^{\top} \big)^{-1} \mathbf{y}} \big(\boldsymbol{\Upsilon}_{\frac{n}{2}} + \theta \mathbf{x} \mathbf{x}^{\top} + \vartheta \mathbf{x} \mathbf{y}^{\top} \big)^{-1} \mathbf{x} \mathbf{y}^{\top} \big(\boldsymbol{\Upsilon}_{\frac{n}{2}} + \theta \mathbf{x} \mathbf{x}^{\top} + \vartheta \mathbf{x} \mathbf{y}^{\top} \big)^{-1} \\
& \hspace*{-5.0cm} = \boldsymbol{\Upsilon}_{\frac{n}{2}}^{-1} - \tfrac{\vartheta + \vartheta^{2} \mathbf{y}^{\top} \boldsymbol{\Upsilon}_{\frac{n}{2}}^{-1} \mathbf{x}}{\rho} \boldsymbol{\Upsilon}_{\frac{n}{2}}^{-1} \left(\mathbf{y} \mathbf{x}^{\top} + \mathbf{x} \mathbf{y}^{\top} \right) \boldsymbol{\Upsilon}_{\frac{n}{2}}^{-1} + \tfrac{\vartheta^{2}\mathbf{y}^{\top} \boldsymbol{\Upsilon}_{\frac{n}{2}}^{-1} \mathbf{y} - \theta}{\rho} \boldsymbol{\Upsilon}_{\frac{n}{2}}^{-1} \mathbf{x} \mathbf{x}^{\top} \boldsymbol{\Upsilon}_{\frac{n}{2}}^{-1} + \tfrac{\vartheta^{2} \mathbf{x}^{\top} \boldsymbol{\Upsilon}_{\frac{n}{2}}^{-1} \mathbf{x}}{\rho} \boldsymbol{\Upsilon}_{\frac{n}{2}}^{-1} \mathbf{y} \mathbf{y}^{\top} \boldsymbol{\Upsilon}_{\frac{n}{2}}^{-1},
\end{split}
\end{equation}
with $\boldsymbol{\Upsilon}_{\frac{n}{2}} := \mathrm{diag} \left(\lambda_{1},\lambda_{3},\ldots,\lambda_{n-1} \right)$, $\mathbf{x}, \mathbf{y}$ given by \eqref{eq:2.4a} and $\rho$ in \eqref{eq:3.13b}. In the same way, supposing \eqref{eq:3.2a} and \eqref{eq:3.2b} (observe that \eqref{eq:3.2c} is $\varrho \neq 0$) we obtain
\begin{equation*}
\big(\boldsymbol{\Delta}_{\frac{n}{2}} + \theta \mathbf{v} \mathbf{v}^{\top} \big)^{-1} = \boldsymbol{\Delta}_{\frac{n}{2}}^{-1} - \tfrac{\theta}{1 + \theta \mathbf{v}^{\top} \boldsymbol{\Delta}_{\frac{n}{2}}^{-1} \mathbf{v}} \boldsymbol{\Delta}_{\frac{n}{2}}^{-1} \mathbf{v} \mathbf{v}^{\top} \boldsymbol{\Delta}_{\frac{n}{2}}^{-1},
\end{equation*}
\begin{equation*}
\begin{split}
&\big(\boldsymbol{\Delta}_{\frac{n}{2}} + \theta \mathbf{v} \mathbf{v}^{\top} + \vartheta \mathbf{v} \mathbf{w}^{\top} \big)^{-1} = \\
& \hspace*{2.0cm} = \big(\boldsymbol{\Delta}_{\frac{n}{2}} + \theta \mathbf{v} \mathbf{v}^{\top} \big)^{-1} - \tfrac{\vartheta}{1 + \vartheta \mathbf{w}^{\top} \big(\boldsymbol{\Delta}_{\frac{n}{2}} + \theta \mathbf{v} \mathbf{v}^{\top} \big)^{-1} \mathbf{v}} \big(\boldsymbol{\Delta}_{\frac{n}{2}} + \theta \mathbf{v} \mathbf{v}^{\top} \big)^{-1} \mathbf{v} \mathbf{w}^{\top} \big(\boldsymbol{\Delta}_{\frac{n}{2}} + \theta \mathbf{v} \mathbf{v}^{\top} \big)^{-1} \\
& \hspace*{2.0cm} = \boldsymbol{\Delta}_{\frac{n}{2}}^{-1} - \tfrac{\theta}{1 + \theta \mathbf{v}^{\top} \boldsymbol{\Delta}_{\frac{n}{2}}^{-1} \mathbf{v} + \vartheta \mathbf{w}^{\top} \boldsymbol{\Delta}_{\frac{n}{2}}^{-1} \mathbf{v}} \boldsymbol{\Delta}_{\frac{n}{2}}^{-1} \mathbf{v} \mathbf{v}^{\top} \boldsymbol{\Delta}_{\frac{n}{2}}^{-1} - \tfrac{\vartheta}{1 + \theta \mathbf{v}^{\top} \boldsymbol{\Delta}_{\frac{n}{2}}^{-1} \mathbf{v} + \vartheta \mathbf{w}^{\top} \boldsymbol{\Delta}_{\frac{n}{2}}^{-1} \mathbf{v}} \boldsymbol{\Delta}_{\frac{n}{2}}^{-1} \mathbf{v} \mathbf{w}^{\top} \boldsymbol{\Delta}_{\frac{n}{2}}^{-1}
\end{split}
\end{equation*}
and
\begin{equation}\label{eq:3.16}
\begin{split}
\big(\boldsymbol{\Delta}_{\frac{n}{2}} + \theta \mathbf{v} \mathbf{v}^{\top} + \vartheta \mathbf{v} \mathbf{w}^{\top} + \vartheta \mathbf{w} \mathbf{v}^{\top} \big)^{-1} &= \\
& \hspace*{-4.5cm} = \big(\boldsymbol{\Delta}_{\frac{n}{2}} + \theta \mathbf{v} \mathbf{v}^{\top} + \vartheta \mathbf{v} \mathbf{w}^{\top} \big)^{-1} + \\
&\hspace*{-3.5cm} - \tfrac{\vartheta}{1 + \vartheta \mathbf{v}^{\top} \big(\boldsymbol{\Delta}_{\frac{n}{2}} + \theta \mathbf{v} \mathbf{v}^{\top} + \vartheta \mathbf{v} \mathbf{w}^{\top} \big)^{-1} \mathbf{w}} \big(\boldsymbol{\Delta}_{\frac{n}{2}} + \theta \mathbf{v} \mathbf{v}^{\top} + \vartheta \mathbf{v} \mathbf{w}^{\top} \big)^{-1} \mathbf{v} \mathbf{w}^{\top} \big(\boldsymbol{\Delta}_{\frac{n}{2}} + \theta \mathbf{v} \mathbf{v}^{\top} + \vartheta \mathbf{v} \mathbf{w}^{\top} \big)^{-1} \\
& \hspace*{-4.5cm} = \boldsymbol{\Delta}_{\frac{n}{2}}^{-1} - \tfrac{\vartheta + \vartheta^{2} \mathbf{w}^{\top} \boldsymbol{\Delta}_{\frac{n}{2}}^{-1} \mathbf{v}}{\varrho} \boldsymbol{\Delta}_{\frac{n}{2}}^{-1} \left(\mathbf{w} \mathbf{v}^{\top} + \mathbf{v} \mathbf{w}^{\top} \right) \boldsymbol{\Delta}_{\frac{n}{2}}^{-1} + \tfrac{\vartheta^{2}\mathbf{w}^{\top} \boldsymbol{\Delta}_{\frac{n}{2}}^{-1} \mathbf{w} - \theta}{\varrho} \boldsymbol{\Delta}_{\frac{n}{2}}^{-1} \mathbf{v} \mathbf{v}^{\top} \boldsymbol{\Delta}_{\frac{n}{2}}^{-1} \\
& \hspace*{5.6cm} + \tfrac{\vartheta^{2} \mathbf{v}^{\top} \boldsymbol{\Delta}_{\frac{n}{2}}^{-1} \mathbf{v}}{\varrho} \boldsymbol{\Delta}_{\frac{n}{2}}^{-1} \mathbf{w} \mathbf{w}^{\top} \boldsymbol{\Delta}_{\frac{n}{2}}^{-1},
\end{split}
\end{equation}
where $\boldsymbol{\Delta}_{\frac{n}{2}} := \mathrm{diag} \left(\lambda_{2},\lambda_{4},\ldots,\lambda_{n} \right)$, $\mathbf{v}, \mathbf{w}$ given by \eqref{eq:2.5a} and $\varrho$ in \eqref{eq:3.13d}. Since the nonsingularity of $\mathbf{H}_{n}$ and $\lambda_{k} \neq 0$, for all $k = 1,\ldots,n$ are sufficient for the both sides of \eqref{eq:3.15} and \eqref{eq:3.16} to be well-defined, conditions \eqref{eq:3.1a}, \eqref{eq:3.1b}, \eqref{eq:3.2a} and \eqref{eq:3.2b} previously assumed can be dropped. Hence, the block diagonalization provided in (a) of Lemma~\ref{lem3} together with 8.5b of \cite{Harville97} (see page $88$) establish the thesis in (a). The proof of (b) is analogous, so that we will omit the details.
\end{proof}

\bigskip

\begin{acknowledgement}
This work is a contribution to the Project UID/GEO/04035/2013, funded by FCT - Funda\c{c}\~{a}o para a Ci\^{e}ncia e a Tecnologia, Portugal.
\end{acknowledgement}


\bigskip

\begin{thebibliography}{2013}

\bibitem{Anderson96} J. Anderson, A secular equation for the eigenvalues of a diagonal matrix perturbation, Linear Algebra Appl. 246 (1996) 49--70.

\bibitem{Asplund59} S.O. Asplund, Finite boundary value problems solved by Green's matrix, Math. Scand. 7 (1959) 49--56.

\bibitem{Bini83} D. Bini and M. Capovani, Spectral and computational properties of band symmetric Toeplitz matrices, Linear Algebra Appl. 52/53 (1983) 99--126.

\bibitem{Bunch78} J.R. Bunch, C.P. Nielsen, D.C. Sorensen, Rank-one modification of the symmetric eigenproblem, Numer. Math. 31 (1978) 31--48.

\bibitem{Demko77} S. Demko, Inverses of band matrices and local convergence of spline projections, SIAM J. Numer. Anal. 14(4) (1977) 616--619.

\bibitem{Fasino88} D. Fasino, Spectral and structural properties of some pentadiagonal symmetric matrices, Calcolo 25(4) (1988) 301--310.

\bibitem{Fischer69} C.F. Fischer and R.A. Usmani, Properties of some tridiagonal matrices and their application to boundary value problems, SIAM J. Numer. Anal. 6(1) (1969) 127--142.

\bibitem{Haley80} S. Haley, Solution of band matrix equations by projection-recurrence, Linear Algebra Appl. 32 (1980) 33--48.

\bibitem{Harville97} D.A. Harville, Matrix Algebra From a Statistician's Perspective, Springer-Verlag, 1997.

\bibitem{Horn13} R.A. Horn, C.R. Johnson, Matrix Analysis (second edition), Cambridge University Press, 2013.

\bibitem{Keeping70} A.J. Keeping, Band matrices arising from finite difference approximations to a third order partial differential, SIAM J. Numer. Anal. 7(1) (1970) 142--156.

\bibitem{Miller81} K.S. Miller, On the inverse of the sum of matrices, Math. Mag. 54(2) (1981) 67--72.

\bibitem{Pissanetsky84} S. Pissanetsky, Sparse Matrix Technology, Academic Press, 1984.

\bibitem{Usmani76} R.A. Usmani, T.H. Andres, D.J. Walton, Error estimation in the integration of ordinary differential equations, Int. J. Comput. Math. 5 (1976) 241--256.
\end{thebibliography}
\end{document}